\documentclass[twocolumn]{my-autart}

\usepackage{color}
\usepackage{booktabs, multirow, makecell}
\usepackage{bm, comment}
\usepackage[lofdepth,lotdepth]{subfig}
\usepackage{amsmath,amsfonts,amssymb}
\usepackage{mathabx,fixmath,algorithm, algorithmic}
\usepackage{pstool}
\usepackage{psfrag}
\usepackage{hyperref}
\usepackage{cite}

\def\x{\mathbold{x}}
\def\k{\mathbold{k}}

\def\I{\mathbold{I}}
\def\K{\mathbold{K}}
\def\Reals{\mathbf{R}}

\def\Q{\mathbf{Q}}

\def\y{\mathbold{y}}

\def\transp{\mathsf{T}}

\usepackage{theorem}

\newtheorem{theorem}{Theorem}
\newtheorem{corollary}{Corollary}
\newtheorem{lemma}{Lemma}

\newtheorem{remark}{Remark}
\newtheorem{assumption}{Assumption}

\newenvironment{proof}{\emph{Proof:}}{\hfill $\blacksquare$}

\newcommand{\ALGO}{\textsf{AGP-UCB}}

\newcommand{\rev}[1]{{\color{black} #1}}
\newcommand{\arev}[1]{{\color{black} #1}}

\newcommand{\cref}[1]{\eqref{#1}}

\allowdisplaybreaks

\usepackage{lipsum}
\usepackage{amsfonts}
\usepackage{graphicx}
\usepackage{epstopdf}
\usepackage{algorithmic}
\ifpdf
  \DeclareGraphicsExtensions{.eps,.pdf,.png,.jpg}
\else
  \DeclareGraphicsExtensions{.eps}
\fi

\usepackage{ifthen}
\newboolean{showcomments}
\setboolean{showcomments}{true}
\newcommand{\andrey}[1]{\ifthenelse{\boolean{showcomments}}
	{ \textcolor{blue}{(Andrey says:  #1)}}{}}
	
\def\endstatement{\hfill$\diamond$}
	
\begin{document}

\begin{frontmatter}

\title{Personalized Optimization with User's Feedback\thanksref{footnoteinfo}} 

\thanks[footnoteinfo]{This paper was not presented at any conferences. Corresponding author A. Simonetto. Tel. +353-87-3583843. 
Fax +353-1-8269903.}

\author[IBM]{Andrea Simonetto}\ead{andrea.simonetto@ibm.com}, 
\author[UCB]{Emiliano Dall'Anese}\ead{emiliano.dallanese@colorado.edu},    
\author[Amazon]{Julien Monteil}\ead{jul@amazon.com},               
\author[NREL]{Andrey Bernstein}\ead{andrey.bernstein@nrel.gov}  

\address[IBM]{IBM Research Ireland, Dublin, Ireland}  
\address[UCB]{University of Colorado Boulder, Boulder, CO, USA}             %
\address[Amazon]{Amazon Web Services, One Burlington Plaza, Burlington Road, Dublin 4, Ireland} 
\address[NREL]{National Renewable Energy Laboratory (NREL), Golden, CO, USA}        









\vskip-2mm
\begin{abstract}
This paper develops an online algorithm to solve a time-varying optimization problem with an objective that comprises a known time-varying cost and an unknown function. This problem structure arises in a number of engineering systems and cyber-physical systems where the known function captures time-varying engineering costs, and the unknown function models user's satisfaction; in this context, the objective is to strike a balance between given performance metrics and user's satisfaction. Key challenges related to the problem at hand are related to (1) the time variability of the problem, and (2) the fact that learning of the user's utility function is performed concurrently with the execution of the online algorithm. This paper leverages Gaussian processes (GP) to learn the unknown cost function from noisy functional evaluation and build pertinent upper confidence bounds. Using the GP formalism, the paper then advocates time-varying optimization tools to design an online algorithm that exhibits tracking of the oracle-based optimal trajectory within an error ball, while learning the user's satisfaction function with no-regret. The algorithmic steps are inexact, to account for possible limited computational budgets or real-time implementation considerations. Numerical examples are illustrated based on a problem related to vehicle platooning. 
\end{abstract}

\end{frontmatter}

\section{Introduction}
\label{sec:introduction}


Optimization is ubiquitous in engineering systems and cyber-physical systems including smart homes, energy grids, and intelligent transportation systems. As an example for the latter, optimization toolboxes are advocated for intelligent traffic light management and for congestion-aware systems for highways, where a networked control system could control vehicles and form a platoon optimized for lowering fuel emission and increasing vehicle spatial density. In many applications involving (and affecting) end-users, optimization problems are formulated  with the objective of striking a balance between given engineered performance metrics and user's satisfaction. Engineered metrics may include, e.g., operational cost and efficiency; on the other hand, metrics to be addressed for the users may be related to comfort (e.g., temperature in a home or building), perceived safety (e.g., distance from preceding  vehicles while driving on a freeway), or simply preferences (e.g., taking a route with the least number of traffic lights).

While engineered performance goals may be synthetized based on well-defined metrics emerging from physical models or control structures, the utility function to be optimized for the users is primarily based on synthetic models postulated for comfort and satisfaction; these synthetic functions are constructed based on generic welfare models, averaged or statistical human-perception models estimated over a sufficiently large population of users' responses, or just simplified mathematical models that make the problem tractable. However, oftentimes, these strategies do not lead to meaningful optimization outcomes. \rev{Synthetic models of users' utilities are difficult to obtain for the associated cost and time of human studies, the data is therefore scarce, biased, and there is a lack of comprehensive theoretical models, see e.g.~\cite{Lepri2017,Bourgin2019}. In addition, (1) generic welfare or averaged  models  may not fully capture preferences, comfort, or satisfaction of individual users, and (2) synthetic utility functions may bear no relevance to a number of individuals and users, that is \emph{personalization} is key (as well-known, e.g., in medicine~\cite{Spaulding2011}).}

This paper investigates time-varying optimization problems with an objective that comprises a \emph{known} time-varying cost and an \emph{unknown} utility function. The known function captures time-varying engineering costs, where time variability emerges from underlying dynamics of the systems.  The second term pertains to the users, and models personalized utility functions; these functions  are in general unknown, and they must be learned concurrently with the solution of the optimization problem. \rev{Learning unknown  utility functions for each  user is key to overcome the limitations of synthetic models, and to fully capture user satisfaction from possibly limited data. }

Typically, three separate timescales are considered in this setting: \emph{(i)} temporal variations of the engineering cost; \emph{(ii)} learning rates for the user's utility function; and, \emph{(iii)} convergence rate of the optimization algorithm. In the paper, we compress these timescales by advocating time-varying optimization  tools to design an online algorithm that tracks time-varying optimal decision variables emerging from time-varying engineering costs within an error bound, while concurrently learning the user's satisfaction with no-regret. We  model the user's satisfaction as a Gaussian process (GP) and learn its parameters from user's feedback~\cite{Srinivas2012,Bogunovic2016}; in particular, feedback is in the form  of noisy functional evaluations of the (unknown) utility function. 
The user's satisfaction is learned by implementing the approximate decisions and measuring the user's feedback to update the user's GP model. 

Overall, the main contributions of the paper are as follows: 

\noindent $\bullet$ The paper extends the works~\cite{Srinivas2012,Bogunovic2016} on GP\arev{s} to handle a sum of a convex engineering cost and an unknown user utility function by considering \emph{approximate} online algorithms; the term ``approximate'' refers to the fact that intermediate optimization sub-problems in the algorithmic steps are not assumed to be solved to convergence due to underlying complexity limits. This is key in many time-varying optimization applications and increasingly important for large-scale  systems~\cite{Hours2014,Hauswirth2017,DallAnese2016,Paternain2018,SPM,Simonetto2020}.

\noindent $\bullet$ We devise an \emph{approximate upper confidence bound algorithm in compact continuous spaces} and provide its regret analysis to solve the formulated problem. As in~\cite{Srinivas2012}, for the time-invariant and squared exponential kernel case we recover a cumulative regret of the form $O^*(\sqrt{d T (\log T)^{d+1}})$, where $d$ is the dimension of the problem and $T$ is horizon\footnote{\hskip-2mm The notation $O^*$ means a big-$O$ result up to polylog factors.}. As in~\cite{Bogunovic2016}, the time-variation of the cost yields an extra $O(T)$ term in the cumulative regret. We explicitly show how the regret depends on the convergence rate of the \arev{optimization algorithm that is used}, and how it is dependent on the choice of the kernel.

\noindent $\bullet$ We further detail the regret analysis to the case where we use a projected gradient method to compute the approximate optimizers of the upper confidence bound algorithm, and we extend the analysis to vanishing cost changes.

\arev{The performance of the proposed algorithm is then demonstrated in a numerical example derived from vehicle platooning, where the inter-vehicle distances are computed in real-time based on both the engineering and user's comfort perspectives.}

Incorporating user's satisfaction in the decision-making process is not a new concept\rev{~\cite{Kahneman1979}}; yet, it is starting to play a major role in emerging data-driven optimization and control paradigms, especially within the context of cyber-physical-social systems or cyber-physical-and-human systems. For example,~\cite{Bae2018} employs control-theoretic techniques to model human behavior and drive systems to suitable working conditions. In this context, human behavior has been captured as a stochastic process~\cite{Pentland1999} and discrete decision models~\cite{McFadden2000}, among others. 

User's satisfaction and preferences have been modelled as a GP in the machine learning community; see, e.g.,~\cite{Chu2005,Houlsby2012}. For an account of Gaussian processes we refer to ~\cite{Rasmussen2006}, while for their use in control we refer the reader to~\cite{Berkenkamp2016,Nghiem2017,Jain2018,Liu2018a}. User's satisfaction and comfort has been taken into account from a control perspective in, e.g., control systems for houses, electric vehicles, and routing~\cite{Oldewurtel2012,Chatupromwong2012,Quercia2014}.

\rev{In all the works mentioned above, the users' functions were modelled based on synthetic models or they were learned a priori (i.e., before the execution of the algorithm). Preferences  are also  learned among a finite number of possibilities. In this paper, we incorporate learning modules in the online optimization algorithm (with the learning and the optimization tasks implemented in a concurrent timescale), we bypass assumptions on discrete preference sets, and we do not utilize synthetic models. The proposed strategy enables fine-grained personalization. }

\rev{With respect to works that use GPs in control and reinforcement learning (RL)\arev{~\cite{Deisenroth2015,Liu2018a,Ghavamzadeh2015,Deisenroth2010,Pinsler2018}}, they typically use GPs to model unknown system dynamics and state transitions; in this paper,  we do not consider system dynamics and state transitions,  since we focus on \arev{a} purely optimization strategy. \arev{Even when the reward function is modeled as \arev{a} GP, the RL machinery is more complex than what we present here; RL-based methods also implicitly assume a timescale separation between optimization and learning, which is not the case in this paper.}}

The techniques and ideas that are expressed in this paper are related to inverse reinforcement learning~\cite{Abbeel2004}\rev{-\cite{Levine2011}} \rev{(even though in those settings one still needs to know what is a ``good'' action via demonstration)}, restless bandit problems~\cite{Slivkins2008} (even though we do not use their machinery and they are typically in discrete spaces, while we are in compact continuous spaces), and (partially) to socially-guided machine learning~\cite{Breazeal2008}. \rev{Ideas in this \arev{paper} are also close in spirit to preference elicitation studies; see, e.g.,~\cite{Huber1993,Blum2004,Weernink2014}}.

Online convex optimization and learning are active research areas in the control and signal processing communities, see e.g.~\cite{Hosseini2016,Ma2017,Nedic2017,Cao2018,Zhou2018,Koppel2018,Shahrampour2018,ElChamie2019,Akbari2019,Epperlein2019}; broadly, this paper aims at expanding this line of work by cross-fertilizing time-varying optimization, online learning, and Gaussian Processes by incorporating user's feedback, and by providing an illustrative application in vehicle control.

\smallskip

{\bf Organization.} The remainder of the paper is organized as follows. Section~\ref{sec:form} presents the mathematical formulation and the proposed approach using an approximate upper confidence bound algorithm. The convergence properties of said algorithm are analyzed in Section~\ref{sec:rb}, along with two specifications of the algorithm in cases where we use a projected gradient method and the cost changes are vanishing. In Section~\ref{sec:num}, we report numerical examples. Proofs of all the results are in the Appendix.  

\section{Formulation and Approach}\label{sec:form}

Consider a decision variable $\x \in D \subset \Reals^N$, and a possibly time-varying objective function $f(\x; t): \Reals^{N} \times \Reals_{+} \to \Reals$ that  has to be maximized. The function $f(\x; t)$, where $t>0$ represents time, is given by the sum of two terms: a concave engineering \rev{welfare} function $V(\x; t): \Reals^{N} \times \Reals_{+} \to \Reals$, and  a user's satisfaction function $U(\x): \Reals^{N} \to \Reals$. The function $V(\x; t)$ is known (or can be easily evaluated at time $t$) \rev{and represents a deterministic function that
evaluates system rewards or engineering costs}; on the other hand, $U(\x)$ is unknown and has to be learned. Accordingly, the problem to be solved amounts to:
\begin{equation}\label{eq:tvp}
\textsf{P}(t): \,\, \max_{\x \in D} f(\x; t) = V(\x;t) + U(\x), \,\, \textrm{ for all } t\geq 0,
\end{equation}
that is, the objective is \arev{to} find an optimal decision $\x^*(t)$ for each time $t$. Here, we implicitly assume that the user's satisfaction function $U(\x)$ does not change in time; however, slow changes of the function do not affect reasoning and technical approach. Furthermore, $U(\x)$ is not necessarily a concave function, thus rendering $\textsf{P}(t)$ a challenging problem even in the static setting. 
The structure of~\eqref{eq:tvp} naturally suggests  that optimal decisions strike a balance between design choices (which may be time-varying, e.g., in tracking problems) and user's preferences, which are often slowly varying and not known beforehand. 

As a first step, consider sampling the problem~\eqref{eq:tvp} at discrete time instances $t_k$, $k=1,2,\ldots$, with $h$ a given sampling period. This leads to a sequence of time-invariant problems as:
\begin{equation}\label{eq:tivp}
\textsf{P}_k:\quad \max_{\x \in D} f(\x; t_k) = V(\x;t_k) + U(\x),
\end{equation}
which we want to solve approximately within the sampling period $h$ to generate a sequence of approximate optimizers $\{\x_k\}_{k\in\mathbf{N}}$ (one for each problem $\textsf{P}_k$) that eventually converges to an optimal decision trajectory $\{\x^*_k\}_{k\in\mathbf{N}}$ up to a bounded error. When only the known engineering \rev{welfare} function $V(\x; t)$ is considered, this tracking problem  has been considered in various prior works; see,  e.g.,~\cite{Paper3,Fazlyab2016,DallAnese2016,Bernstein19,Dixit2018,SPM,Simonetto2020} and pertinent references therein. A key difference in the  setting proposed in this paper is that we construct approximate optimizers  \emph{concurrently} with the learning of the (unknown) user's function $U(\x)$. The main operating principles of the algorithm to  be explained shortly are, qualitatively, as follows: (i) at time $t_k$,  an approximate optimizer $\x_{k}$ is computed based on a partial knowledge of $U(\x)$; (ii) the optimizer is  implemented, and it generates some ``feedback'' from the user in the form of, for example, $y_k = U(\x_{k}) + \varepsilon$, where $\varepsilon$ is noise; and, (iii) $y_k$ is collected and utilized to ``refine'' the knowledge of $U(\x)$. At time $t_{k+1}$, the process is then repeated. 

To this end, the paper leverages a Gaussian process (GP) model for the unknown user function $U(\x)$. Such non-parametric model is advantageous in the present setting because of (1) the simplicity of the online updates of both mean and covariance; (2) the inherent ability to handle asynchronous and intermittent updates (which is an important feature in user's feedback systems); and, (3) the implicit and smooth handling of measurement (i.e., feedback) noise. Accordingly, let $U(\x)$ be specified by a GP with mean function $\mu(\x) := \mathbf{E}[U(\x)]$ and covariance (or kernel) function $k(\x, \x') := \mathbf{E}[(U(\x) - \mu(\x))(U(\x')- \mu(\x'))]$. We assume bounded variance; i.e., $k(\x,\x) \leq 1, \x \in D$. For GPs not conditioned on data, we assume without loss of generality that $\mu \equiv 0$, i.e., GP$(0,k(\x,\x'))$. \rev{In this context, the user function $U(\x)$ is assumed to be a sample path of GP$(0,k(\x,\x'))$}. Let $A_n = \{\x_1\in D, \ldots, \x_n\in D\}$ be a set of $n$ sample points and let $y_i = U(\x_i) + \varepsilon_i$, $\varepsilon_i \sim N(0, \sigma^2)$ i.i.d. Gaussian noise, be the noisy measurements at the sample points $\x_i$, $i = 1, \ldots,n$. Let $\y_n = [y_1, \ldots, y_n]^\transp$. Then the posterior distribution of $(U(\x) | A_n, \y_n)$ is a GP distribution with mean $\mu_n(\x)$, covariance $k_n(\x,\x')$, and variance $\sigma_n^2(\x)$ given by:
\begin{eqnarray}
\mu_n(\x) &=& \k_n(\x)^\transp(\K_n + \sigma^2 \I_n)^{-1} \y_n \label{eq:mean} \\
k_n(\x,\x') &=& k(\x,\x') \!-\! \k_n(\x)^\transp(\K_n + \sigma^2 \I_n)^{-1}\k_n(\x')
\color{white}{aaaa}\color{black} \label{eq:cov} \\
\sigma_n^2(\x) &=& k_n(\x,\x). \label{eq:var}
\end{eqnarray}
where $\k_n(\x) := [k(\x_1,\x), \ldots, k(\x_n,\x)]^\transp$, and $\K_n$ is the positive definite kernel matrix $[k(\x,\x')]_{\x,\x' \in A_n}$.

With this in place, in order to approximately solve the sequence of optimization problems $\{\textsf{P}_k\}$ (where we remind that the objective function is partially unknown), we utilize an online approximate Gaussian process upper confidence bound (\ALGO) algorithm as described \rev{in the \ALGO\ table in the next page.\\} 

\begin{figure}
\hrule

\hrule

\vspace{.1cm}

\noindent \textbf{\ALGO~algorithm}

\vspace{.1cm}

\hrule 

\vspace{.1cm}

\noindent Initialize $\mu_0(\x)$, $\sigma_0(\x)$ from average user's profiles. Set \rev{the optimization counter $k$ and the data \arev{counter} $n$ to $1$.} 
\vskip1mm

\hrule 
\vskip1mm

\noindent  At each time $t_k$, perform the following steps {\bf[S1]}--{\bf[S3]}: 
\vskip1mm

\noindent {\bf[S1]} Set 
\begin{equation}\label{eq:ucb}
\hat{U}_n(\x) = \mu_{n-1}(\x) + \sqrt{\beta_n} \sigma_{n-1}(\x),
\end{equation}
\rev{The confidence parameter $\beta_n$ is set as in Th.~\ref{th:Regret};}

\vskip1mm
\noindent {\bf[S2]} Find a possibly approximate optimizer:
\begin{equation}\label{eq:optimizer}
\x_{k} \approx \textrm{arg}\max_{\x\in D} \{ V(\x; t_{k}) \!+\! \hat{U}_n(\x)\},
\end{equation}
by running a finite numbers of steps of a given algorithmic scheme, and implement $\x_{k}$;

\vskip1mm
\noindent {\bf[S3]} \rev{Ask user's feedback on $\x_{k}$}: 
\vskip1mm

\hskip.5cm \begin{minipage}{0.4\textwidth}\rev{\textbf{If} (feedback is given): collect
$$
y_{n} = U(\x_{k}) + \varepsilon_n,
$$
and perform a Bayesian update to obtain $\mu_{n}(\x)$ and $\sigma_{n}(\x)$ according to \eqref{eq:mean}-\eqref{eq:var};  set $n$ to $n$+$1$.}
\end{minipage}

\vskip1mm

\hskip.5cm \begin{minipage}{0.4\textwidth} \rev{\textbf{Else} (if the feedback $y_{n}$ is not received): 
keep $\mu_{n}(\x)$ and $\sigma_{n}(\x)$ at their current values.} \end{minipage}

\vspace{.1cm}

\hrule 

\hrule 
\end{figure}

\subsection{\rev{Description of \ALGO.}}

\rev{We now describe the main steps of \ALGO. First of all, notice that there are two indexes $k$ and $n$; $k$ represents the time/optimization counter (i.e., the time index for the execution of the step~\eqref{eq:optimizer}), while $n$ is the data counter (see step [S3] in \ALGO). Time and optimization steps are fused into $k$, since one step of the algorithm is performed at each time. The data counter $n$ does not generally coincide with $k$, since feedback may not be given at each algorithmic step. After initialization,} in~\eqref{eq:ucb}, a proxy for the unknown function $U(\x)$ is built using an upper confidence bound. Upper confidence bound methods 
are popular in stochastic bandit settings; see, e.g.,~\cite{Srinivas2012,Bubeck2012}. The aim of an upper confidence bound method is to trade off exploitation (areas with high mean but low variance) and \arev{exploration} (areas with low mean but high variance). \rev{The upper confidence bound depends on the parameter $\beta_n$ which is set as described in Th.~\ref{th:Regret}. }

Step~\eqref{eq:optimizer} is utilized to find $\x_k$ based on the (current) upper confidence bound \rev{and the new function $V(\cdot; t_k)$}. In~\cite{Srinivas2012}, it is assumed that~\eqref{eq:optimizer} can be solved to optimality; on the other hand, we consider a setting where only an \emph{approximate} optimizer of~\eqref{eq:optimizer} can be obtained at each time $t_k$. In particular, we consider a case where it might not be possible to execute the algorithm until convergence to an optimal solution within a period of time $h$~\cite{Paper3,Fazlyab2016,DallAnese2016,Bernstein19,Dixit2018}; this is the case where, for example, the sampling period $h$ is too short (relatively to the time required to perform an algorithmic step and the number of iterations required to converge) or the problem is computationally demanding. Thus,  
denoting as $\mathcal{M}$ the map of a given algorithmic step (e.g., gradient descent, proximal method, etc.) and assuming that one can perform $N_s$ steps within an interval $h$, $\x_{k}$ is obtained as: 
\begin{equation}\label{eq:met}
\x_{k} = \underbrace{\mathcal{M}\circ \cdots \circ \mathcal{M}}_{N_s \textrm{~compositions}}\circ \varphi_{nk}(\x_{k-1}).
\end{equation}
where $\varphi_{nk}(\x)$ is defined as:
\begin{equation} \label{eq:phi}
\varphi_{nk}(\x):=V(\x; t_{k}) + \mu_{n-1}(\x) + \sqrt{\beta_n} \sigma_{n-1}(\x). 
\end{equation}
For example, if $\mathcal{M}$ is the map of a gradient method, then~\eqref{eq:met} implies that one runs $N_s$ gradient steps\footnote{\arev{The choice of the algorithm will be dictated by Assumption~\ref{as.method} and it will influence the overall performance.}}; this example will be explained shortly in Section~\ref{sec:example_gradient}. See also the works~\cite{Paper3,Fazlyab2016,DallAnese2016,Bernstein19,Dixit2018} 
(and pertinent references therein) for the case of $N_s = 1$. 
From a more utilitarian perspective, this setting allows one to allocate computational resources to solve~\eqref{eq:optimizer} parsimoniously; spending resources to  solve~\eqref{eq:optimizer} to optimality might not provide performance gains since  the user's function is \emph{not known accurately}. \rev{We remark here that not solving~\eqref{eq:optimizer} \arev{to} optimality and still guaranteeing convergence results is one of the key contribution\arev{s} of this paper. }

Finally, step {\bf[S3]} involves the gathering of the user's feedback in the form of a measurement of $U(\x_k)$. This feedback is in general noisy (e.g., it may be collected by sensors or be quantized), and it could be intermittent in the sense that it might not be available at every time step (this explains why we utilize the subscript $k$ for the temporal index and $n$ for the updated of $\hat{U}_n(\x)$; of course, $k = n$ if $y_{n}$ is collected at each time $k$). If the feedback is available, the GP model is updated via~\eqref{eq:mean}-\eqref{eq:cov}, otherwise not. \rev{This is also further explained in Figure~\ref{fig.streams}, where the temporal index $k$ is dictated by the optimizer counter, while the data counter $n$ is dictated by how many feedback data we get from the users. Since $U$ is time-invariant, it does not matter when the feedback is received, as long as \arev{it is} correctly labeled.}


\begin{figure}
\includegraphics[width=0.5\textwidth]{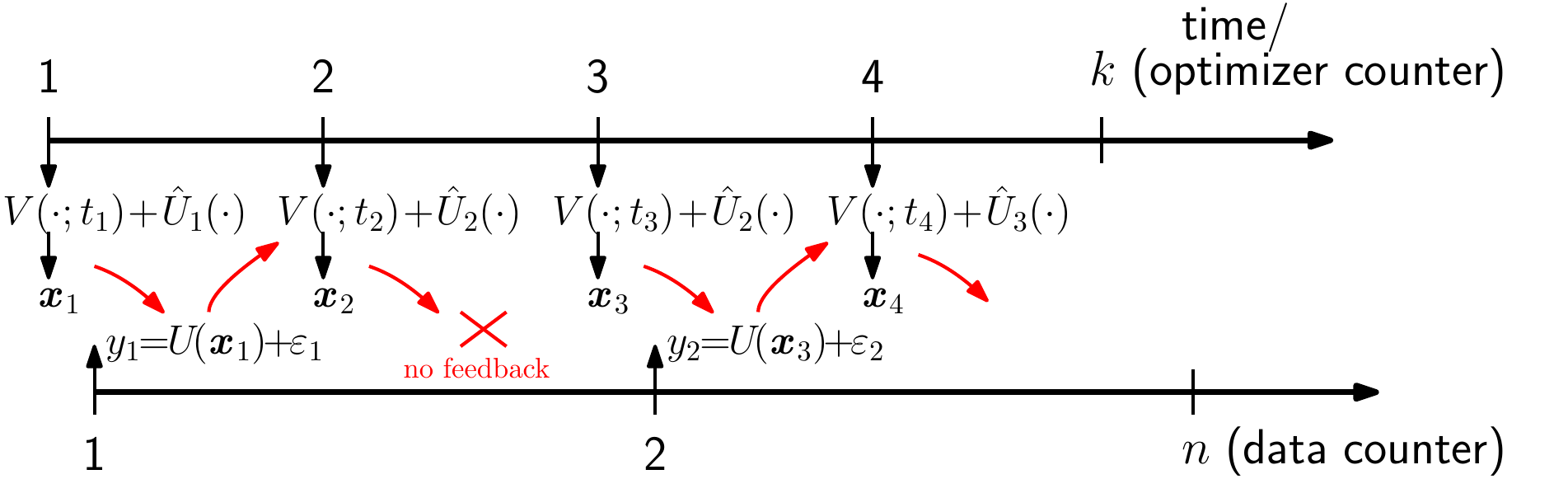}
\caption{\rev{Graphical explanation of \ALGO. At each $t_k$, we compute an approximate $\x_k$ based on the engineering function $V(\cdot; t_k)$ and the user's function $\hat{U}_n$. Based on $\x_k$, we ask feedback to the user, if given: we update $\hat{U}_n$ and increase the data counter, otherwise we keep $\hat{U}_n$ as is.}}
\label{fig.streams}
\end{figure}

\begin{remark}
\rev{\ALGO\ assumes the possibility to acquire the feedback $y_n$ (associated with a given decision variable $\x_k$). Again, feedback may be infrequent, and does not have to happen at every time $k$. Realistic examples of applications include, e.g., a platooning task (see Section~\ref{sec:num}) with sensors in the seat of the vehicle determining stress levels; a personalized medicine example, where the effects of a treatment can be monitored by user's feedback via dedicated mobile app; a building temperature control, where the user's satisfaction can be evaluated by wearables via perspiration levels. } 
\end{remark}

\section{Regret Bounds}\label{sec:rb}

In this section, we establish cumulative regret bounds for the  \ALGO~algorithm, \rev{which \arev{depend} on the amount of feedback we receive}. A critical quantity affecting these bounds is the maximum information gain $\gamma_T$ after $T$ feedback rounds, which is defined as
\begin{equation}
\gamma_T := \max_{A \subset D: |A| = T} \,\,\frac{1}{2}\log|{\bf I} + \sigma^{-2} \K_A|,
\end{equation} 
with $\K_A = [k(\x,\x')]_{\x,\x' \in A}$, and where $A$ represents the set of $T$ points $\x\in D$, which maximizes the expression [cf.~\cite{Srinivas2012}].

The regret bounds are of the form $O^*(\sqrt{T \beta_T \gamma_T}) + O(T)$, where the first term is sub-linear and depends on how fast one can learn the unknown user's profile, while the second term is due to the time drift (i.e., temporal variability) of the design function. These regrets bound resemble~\cite{Srinivas2012} for the time invariant case and \cite{Bogunovic2016} for the time-varying one\footnote{We leave for future research the use of more sophisticated methods that learn, predict, or bound the variations of the cost function and might deliver no-regret results also in the time-varying setting, see~\cite{Besbes2013,Jadbabaie2015}. }. The main proof techniques that we \rev{employ are a combination of} GP regret results, convex analysis, and Kernel ridge regression theory. 

\subsection{Main result}

The following assumptions are imposed.

%
%
\begin{assumption}\label{as.fun}
The function $-V(\x;t)$ is convex and $L$-strongly smooth over $D$, uniformly in $t$.  
\end{assumption}



%
%
\begin{assumption}\label{as.u}
Let $D \subseteq [0, r]^d$ be compact and convex, $d \in \mathbf{N}$, $r>0$. Let the kernel $k(\x,\x')$ satisfy the following high probability bound on the derivatives of the GP sample paths $U$: for some constants $a, b >0$:
\begin{equation*}
\mathbf{Pr}\left\{\sup_{\x \in D} |\partial U/ \partial x_j| > M\right\} \leq a \mathrm{e}^{-(M/b)^2}, \quad j = 1,\ldots,d.
\end{equation*}
\end{assumption}

%
%
\begin{assumption}\label{as.tv}
The changes in time of function $V(\x;t)$ are  bounded, in the sense that at two subsequent sampling times $t_k$ and $t_{k-1}$, the following bound holds
$$
\Delta_k := \max_{\x\in D} |V(\x;t_{k}) - V(\x; t_{k-1})| \leq \Delta.
$$
for a given $\Delta < \infty$.
\end{assumption}

%
%
\begin{assumption}\label{as.method}
Recall the definition of $\varphi_{nk}(\x)$ in \eqref{eq:phi} and let its maximum be $\varphi_{nk}^*$. There exists an algorithm  $\mathcal{M}$ with linear convergence \arev{that, when applied to $\varphi_{nk}(\x_{k-1})$, yields a point $\x_{k}$ for which} 
$$
\varphi_{nk}^* - \varphi_{nk}(\x_{k}) \leq \eta (\varphi_{nk}^* - \varphi_{nk}(\x_{k-1})), \quad \eta<1.
$$ 
\end{assumption}

\begin{assumption}\label{as.behave} The kernel $k(\cdot)$ is not degenerate; the mean and variance are well-behaved, so that can be represented by generalized Fourier series of the kernel basis.
\end{assumption}

Most of the assumptions are standard in either time-varying optimization or Gaussian process analysis; see, e.g., \cite{Srinivas2012,Paper3}. 
From Assumption~\ref{as.fun}, one has that  the following holds for any $\x, \y \in D$ and $t_k$: 
\begin{equation*}
-V(\x; t_k) + V(\y; t_k) \leq -\nabla_{\y} V(\y; t_k)(\x - \y) + L/2 \|\x-\y\|_2^2 . 
\end{equation*}
Further, since $D$ is compact,
one has that: 
\begin{equation}\label{eq:lip}
-V(\x; t_k) + V(\y; t_k) \leq d D_g \|\x - \y\|_{\infty} + L/2 \|\x-\y\|_2^2
\end{equation}
where $D_g$ is defined (uniformly in $t$) as:
\begin{equation}\label{eq:dg}
D_g := \max_{\y \in D} \|-\nabla_{\y} V(\y; t)\|_{\infty} \, .
\end{equation} 

Assumption~\ref{as.u} is standard in the analysis of GPs (see, e.g.,~\cite{Srinivas2012}); it holds true for four-time\arev{s} differentiable stationary kernels, such as the widely-used squared exponential and (some) Mat\'ern kernels, \rev{see~\cite[Eq.~(2.33)]{Azaies2009}, or equivalently~\cite[Theorem~5]{Ghosal2006}}. This assumption is needed when $D$ is compact in order to ensure smoothness of the GP samples \rev{(Cf. Appendix~\ref{ap:ab} for how to compute $a,b$)}. 

Assumption~\ref{as.tv} is required in time-varying settings to bound the temporal variability of the cost function; specifically, it presupposes that the  decision $\x$ yields similar function values for $V$ at two subsequent time (with their difference being upper bounded by $\Delta$). This assumption is also common in online optimization and machine learning; see, e.g.,~\cite{Besbes2013,Jadbabaie2015}.  

\arev{Assumption~\ref{as.method} constraints the possible algorithms that are employed to approximately solve \eqref{eq:optimizer} in {\bf [S2]}. 
T}wo cases are in order based on whether the assumption holds locally around the optimal trajectory (in the paper, the term optimum refers to  the global one) or  globally. In the first case,  results will be valid around the optimal trajectory (provided that the algorithm is started  close enough). In the second case, results will be valid in a general sense. Since the conditions are the same in both cases, the paper will hereafter focus on the second case. In a convex setting, an algorithm  $\mathcal{M}$ with $Q$-linear convergence can be found when the function is strongly convex and strongly smooth;  when the function $-\varphi_{nk}(\x)$ is nonconvex, results are available for special cases. A noteworthy example (which is also explored in the simulation results) is when $-\varphi_{nk}(\x)$ has a Lipschitz-continuous gradient and it satisfies the Polyak-Lojasiewicz (PL) inequality; in this case, the (projected) gradient method has a global linear converge rate and Assumption~\ref{as.method} is satisfied~\cite{Karimi2016}. This case implies that the function $-\varphi_{nk}(\x)$ is invex (i.e., any stationary point must be a global minimizer).
\footnote{Other cases in which Assumption~\ref{as.method} is verified for the general nonconvex case could be found by using recent global convergence analysis of ADMM~\cite{Wang2018,Themelis2018}, and left for future research. Note that if one has a $C/k$ or $C/\sqrt{k}$ method, then one can transform it to a linear converging method by running at least $k>C/\eta$ or $k>(C/\eta)^2$ iterations, respectively.}

%
%
With this in mind, one could substitute Assumption~\ref{as.method} with the following (more restrictive) assumption.

\begin{assumption}\label{as.method.1}
The function $-\varphi_{nk}(\x)$ in \eqref{eq:phi} has a Lipschitz continuous gradient with coefficient $\Theta$ and satisfies the Polyak-Lojasiewicz (PL) inequality as,
\begin{equation}\label{pl:our}
\frac{1}{2}\mathcal{D}(\x, c) \geq \kappa (\varphi_{nk}^*-\varphi_{nk}(\x)), \quad \textrm{with}
\end{equation}
\begin{equation}\label{pl:D}
\mathcal{D}(\x, c) := - 2 c \min_{\y \in D} \left\{ -\nabla_{\x}\varphi_{nk}(\x)^\transp (\y - \x) + \frac{c}{2}\|\y - \x\|^2\right\},
\end{equation}
over $\x \in D$, for some $\kappa >0$, $c>0$, uniformly in time (i.e., $\forall \, k$). 
\end{assumption}

Assumption~\ref{as.method.1} can be also required in a local sense only (around the optimal trajectory), if necessary\footnote{Condition~\eqref{pl:our} imposes that the gradient of the cost function is properly lower bounded. If $D$ is the whole space, $\mathcal{D}$ reduces to $\|\nabla \varphi_{nk}(\x)\|^2$ and the condition becomes of easier interpretation. Condition~\eqref{pl:our} also implies invexity. See~\cite{Karimi2016, Roulet2017} for detailed discussions. }. 

Assumption~\ref{as.behave} is a mild technical assumption needed to determine the learning rates of the proposed method; see also~\cite{Rasmussen2006,Seeger2008,vanderVaart2008,Yang2017}. The non-degeneracy of the kernel holds true for squared exponential kernels~\cite{Zhu1998}. If the assumption is not satisfied, then one would learn a smoother version of $U$. 

\rev{Our theoretical regret analysis will be based on the amount of feedback points $n$ that we receive. From data point $n-1$ to data point $n$, a number of time/optimization steps $k$ may have passed. However,} to simplify the exposition and the notation, hereafter we set  the time index $k$ and \rev{data index} $n$ in \ALGO\ to be the same; that is, $k=n$ (\rev{and we receive feedback at every time $k$/for every approximate decision variable $\x_k$}). This is done without loss of generality, since $k=n$ represents  the worst case  for the regret analysis (\rev{from the data standpoint}); in fact, for $k=n$ the time scales of the variation of $\hat{U}_n$ and the convergence of the optimization algorithm are the same, which causes coupling of the two. For $k\gg n$, the convergence of the optimization algorithm is basically achieved for any $n$ and we are back to timescale separations as in~\cite{Srinivas2012}. \rev{We finally notice that if we have, say, at least $p>0$ optimization/time steps within each data point and at most $q\geq p$, the results here presented will be valid substituting $\eta$ with $\eta^p$ and $\Delta$ with $\Delta (1-\eta^q)/(1-\eta)$. }

\rev{Let then $k = n$, without loss of generality.}

Define the cumulative regret $R_T$ as: 
\begin{equation*}
R_T := \sum_{n=1}^T f^*_n - f(\x_{n};t_n), 
\end{equation*}
where $f^*_n$ is the maximum of function $f(\cdot; t_n)$ at time $t_n$ provided by an oracle\footnote{\rev{The definition of regret here is compatible with the one used in~\cite{Srinivas2012, Bogunovic2016} in the static and time-varying case. However, $R_T$ could be slightly better interpreted as dynamic cumulative regret against a time-varying comparator~\cite{Jadbabaie2015}, or cumulative objective tracking error.
}}.  The following result holds for the regret.


\begin{theorem}[Regret bound]\label{th:Regret}
Let Assumptions~\ref{as.fun}--\ref{as.behave} hold. Pick $\delta \in (0,1)$ and set the parameter $\beta_n$ as
$$
\beta_{n} = 2 \log(2 n^2 \!\pi^2 \!/(3\delta)) + 2 d \log(d n^2 b r \sqrt{\log(4 d a /\delta)}),
$$ 
for $n\geq 1$, where $a,b,d$, and $r$ are the parameters defined in Assumption~\ref{as.u}. Running \ALGO\ with $\beta_{n}$ for a sample $U$ of a GP with mean function zero and covariance function $k(\x,\x')$, \rev{and considering $k=n$}, we obtain a regret bound of $O^*(\sqrt{d T \gamma_T}) + O(T)$ with high probability. In particular, 
\begin{equation}
\mathbf{Pr}\Big\{R_T \leq \sqrt{C_1 T \beta_T \gamma_T} + C_2  + O^*(1)+ G_T \Big\} \geq 1-\delta, 
\end{equation}
where  $C_1 = 8/\log(1+\sigma^{-1})$,
\begin{equation}
C_2 = \frac{2 D_g }{b \sqrt{\log(2 d a /\delta)}}+\frac{2 L}{2 d b^2 {\log(2 d a /\delta)}} + 2,
\end{equation} 
and,
\begin{equation}
G_T := 2 \sum_{n=1}^T \sum_{z=1}^{n-1} \eta^z \Delta_{n-z+1} \leq 2 \Delta \eta T/(1\!-\!\eta).
\end{equation} 
\endstatement
\end{theorem}

\emph{Proof.} See Appendix~\ref{ap:th12}. 

\smallskip

Theorem~\ref{th:Regret} asserts  that the average regret $R_T/T$ converges to an error bound with high probability. It can be noticed that the error bound is a function of the variability of the function $V(\x; t)$; if  $V(\x; t)$ is time-invariant, then a no-regret result can be recovered. The term  $\sqrt{C_1 T \beta_T \gamma_T} + 2$ in the bound can be found in the result of~\cite[Theorem~2]{Srinivas2012} too; in particular, the term $C_1$ depends only  on the variance of the measurement noise. On the other hand, the term $C_2-2$ is due to the time-varying concave function $V(\x; t)$, and it explicitly shows the linear dependence on the Lipschitz constant $L$ of $V(\x; t)$ as well as  on the bound $D_g$ on the gradient of $V(\x; t)$. The term $G_T$ is also due to the time variation of $V(\x; t)$. 

The term $O^*(1)$ emerges from the approximation error in the step {\bf [S2]} of the algorithm and it is key in our analysis; that is, \emph{because only a limited number of algorithmic steps $\mathcal{M}$ are performed}, instead of running the optimization algorithm to convergence. It is also due to the fact that the function $\hat{U}_n$ changes every time a new measurement is collected. Define the learning rate error $\ell_n$ as
$$
\ell_n = \max_{\x \in D} |\hat{U}_n(\x) - \hat{U}_{n-1}(\x)| .
$$
Then, the error term  $O^*(1)$ comes from the sum
\begin{equation}\label{eq:nl}
 \sum_{n=1}^T \sum_{z=1}^{n-1} \eta^{z} \ell_{n-z+1} \leq \frac{1}{1-\eta} O^*(1) = O^*(1),
\end{equation}
whose proof is deferred in the Appendix. From~\eqref{eq:nl}, we can see that if the step {\bf [S2]} is carried out exactly, $\eta \to 0$, then this term vanishes. On the other hand, if $\eta>0$, then the error is weighted by the changes in the surrogate function $\hat{U}_n(\x)$. Because of  the Bayesian update, $\ell_n$ converges fast enough so that its cumulative error is constant. 

The following result pertains to squared  exponential  kernels.

\begin{theorem}[Squared exponential kernels]\label{th:Total}
Under the same assumptions of Theorem~\ref{th:Regret}, consider a squared exponential kernel. Then $\gamma_T = O((\log T)^{d+1})$ and the cumulative regret is of the order of $O^*(\sqrt{d T (\log T)^{d+1}}) + O(T)$.
\endstatement
\end{theorem}
\emph{Proof.} See Appendix~\ref{ap:th12}.

\smallskip

Theorem~\ref{th:Total} is a customized version of Theorem~\ref{th:Regret} for a squared exponential kernel. Other special cases can be derived for other kernels as in~\cite{Srinivas2012}, but are omitted here due to space constraints. In the following, we exemplify  the results of Theorem~\ref{th:Total} for the case where  the algorithmic map $\mathcal{M}$ represents a projected gradient method. 

\subsection{Example: \ALGO~with projected gradient method}
\label{sec:example_gradient}

Consider a projected gradient method, applied to a time-varying problem where Assumption~\ref{as.method.1} holds. Consider further the case where only one step of the projected gradient method can be performed in {\bf [S2]} (i.e., $N_s = 1$); then, denoting as $\Pi_{D}[\mathbf{y}] := \arg \min_{\x \in D} \|\mathbf{y} - \x \|_2$  the projection operator, in this case {\bf [S2]} is replaced with: 

\vspace{.1cm}

\noindent {\bf [S2']} Update  $\x_{k}$ as:
\begin{equation}\label{eq:optimizer:pg}
\hspace{-.2cm} \x_{k} \!=\! \Pi_{D} \!\left[\x_{k-1} \!+\! \alpha\Big(  \nabla_{\x} V(\x_{k-1}; t_{k}) + \nabla_{\x} \hat{U}_n(\x_{k-1})\Big)\!\right]
\end{equation}
and implement $\x_{k}$. 

\vspace{.1cm}

Step {\bf [S2']} encodes a projected gradient method on the function $\varphi_{nk}(\x)$ with stepsize $\alpha$. Note that for the Bayesian framework, the derivative $\nabla_{\x} \hat{U}_n(\x_{k-1})$ is straightforward to compute and no approximation has to be made; this is in contrast with online bandit methods where the gradient is estimated from functional evaluations (see e.g.,~\cite{Flaxman2005,AgarwalXiao10,chen2018bandit} and references therein); this is one of the \arev{strengths} of Bayesian modelling, which however comes at the cost of an increase in computational complexity due to the GP updates. 

Under Assumption~\ref{as.method.1} and with the choice of stepsize $\alpha \leq 1/\Theta$, a suitable extension of the results in~\cite{Karimi2016} (considering the proximal-gradient method with $g$ being the indicator function, see Appendix~\ref{ap:ka}) yields the following convergence result for  the iteration~\eqref{eq:optimizer:pg}: 
\begin{equation}\label{eq:pl}
\varphi_{nk}^* - \varphi_{nk}(\x_k) \leq (1- \alpha \kappa)(\varphi_{nk}^* - \varphi_{nk}(\x_{k-1})).
\end{equation}

In this context, the following corollary is therefore in place. 

\begin{corollary}\textbf{\emph{[\ALGO~with projected gradient method]}}\label{th:pgm}
Consider the modified Step {\bf [S2']} as in \eqref{eq:optimizer:pg}, with $\alpha \leq 1/\Theta$. Under the same assumptions of Theorem~\ref{th:Regret}, but with Assumption~\ref{as.method} replaced with Assumption~\ref{as.method.1} with $c = 1/\alpha$,  Theorem~\ref{th:Regret} holds with the specific value $\eta = 1-\alpha\kappa$.
\endstatement
\end{corollary} 
\emph{Proof.} See Appendix~\ref{ap:ka}.

\subsection{Vanishing Changes}

Consider  the case where $\Delta_k$ in Assumption~\ref{as.tv} vanishes in time; that is, $\Delta_k \to 0$ as $k \to \infty$. This case is important when the variations in the engineering cost function eventually vanish (for example, if the cost function is derived by a stationary process that is learned while the algorithm is run, as typically done in online convex optimization~\cite{Shalev-Shwartz2012}). 

\begin{theorem}[Vanishing changes]\label{th:vc}
Under the same assumptions of Theorem~\ref{th:Regret}, if $\Delta_k \to 0$ as $k \to \infty$, then $G_T$ in Theorem~\ref{th:Regret} can be upper bounded by a sublinear function in $T$ and we obtain a no-regret result. 

Furthermore, if $\Delta_k$ decays at least as $O(1/\sqrt{k})$, then the result of Theorem~\ref{th:Regret} on the regret $R_T$ is indistinguishable from the static result of~\cite[Theorem~2]{Srinivas2012} in a $O^*$ sense. 
\endstatement
\end{theorem}
\emph{Proof.} See Appendix~\ref{ap:th3}.




\section{Numerical evaluation}\label{sec:num}

This section considers an  example of application of the proposed framework in a vehicle platooning problem, \rev{whereby vehicles are controlled to maintain a given reference distance between each other and follow a leader vehicle}, inspired by~\cite{monteil2018,Stuedli2018,monteil2019}, and it provides illustrative numerical results. This problem includes all the modeling elements discussed in the paper, and its real-time implementation requirements are aligned with the design principles of the proposed framework.   

Consider then $m+1$ automated vehicles that are grouped in a platoon. The leading vehicle is labeled as $0$, while the vehicles following the leading one are indexed with increasing numbers from $1$ to $m$. The platoon leading and desired velocity is $v(t)$, while the inter-vehicle distances are denoted as $d_i(t)$ for $i = 1, \ldots, m$ (cf.~Fig.~\ref{fig.0}). Consider the problem of deciding which are the best inter-vehicle distances such that they are as close as possible to some desired values that are  dictated by road, aerodynamics considerations, and platoon's speed,  while being \emph{comfortable} to the car riders; e.g. the automated vehicles distances are not too different than the distances that  users would naturally adopt in human-driven vehicles.

Let $\x := [d_1, \ldots, d_m]^\transp$, and denote as $\bar{\x}$ the time-varying vector of distances that one would obtain by considering only the engineering cost. Then, the problem considered in this section is of the form:
\begin{equation}\label{prob-platoon}
\textsf{P}(t):\quad \max_{\x \in D}\  - \frac{1}{2}\|\x - \bar{\x}(t)\|^2_{\Q} + \gamma \sum_{i=1}^m U_i(\x),
\end{equation}
where $D$ is a compact set representing allowed distances between vehicles; $U_i$ is the function capturing the  ``comfort'' of user $i$; and,  $\|\y\|^2_{\Q}:= \y^\transp \Q \y$ is the weighted norm based on the positive definite matrix $\Q$. We further consider the case $\gamma  = 1$, and $\Q$ being not diagonal (this way, the decision variables are coupled). We also set $U_i(\x) = U_i(d_i)$; that is, the comfort of the $i$-th user depends only on the distance with the preceding vehicle $i - 1$ (however, more general models can be adopted). \rev{The number of vehicles following the leading one is set to $m=2$.}

\begin{figure}[b!]
\centering
\includegraphics[width = 0.475\textwidth]{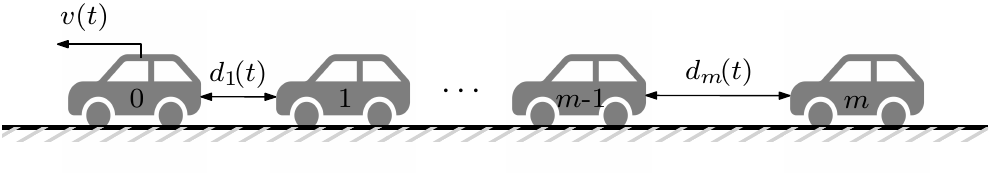}
\caption{\rev{A vehicle platooning example, whereby vehicles $1, \ldots, m$ are controlled to maintain a given time-varying distance between each other $d_i(t)$ and follow the leader vehicle $0$ moving at speed $v(t)$. In the simulation example, $m = 2$. }}
\label{fig.0}
\end{figure} 

 \rev{Functions $U_i(d_i)$ are sample paths from a GP$(0, k(\x, \x'))$ with a squared exponential kernel with length scale parameter $\ell = 1$, i.e.,
$k(\x, \x') = \exp\left(-\frac{1}{2 \ell^2} \|\x-\x'\|^2\right)
$, selected to mimic log-normal functions, and specifically, functions of the form 
$\exp(-\log(d_i)^2/\xi_i^2)/{(\xi_i d_i)} 
$, with $\xi_1 = 0.6$ and $\xi_2 = 0.7$}; the maximizer of $\sum_{i=1}^m U_i(\x)$ does not coincide with the one of the engineering function to avoid a trivial solution. Further, the functions $U_i$'s are different for each vehicle. Using log-normal functions for users' comfort was motivated, see e.g.~\cite{greenberg1966log}, by observing that inter-vehicle times and distances follow log-normal distributions. Intuitively, smaller distances are more critical than larger ones, and there is a given distance  after which the comfort decreases since the users feel that they are too slow relative to the preceding vehicle. In the simulations, we learn functions $U_i$ via a GP with the same kernel. 

It is important to notice that, with the selected parameters, the approximate function $-\varphi_{nk}(d_i)$ is invex, it has Lipschitz continuous gradient, and it also satisfies the Polyak-Lojasiewicz inequality over $D$. With this in mind, we can readily apply a projected gradient method with linear convergence. \rev{Here, we apply one step of projected gradient method per time step $k$.}   

For the numerical tests, we set the probability $\delta$ to $\delta = 0.1$,  and the step size for the gradient method to approximately solve~\eqref{eq:optimizer} is set to $\alpha = 0.1$. The set $D$ is $D = [0,1]^m$ (distances are properly scaled so that the set $[0,1]$ maps to a real distance of $[0,3]$). \rev{We compute the constants $a, b$ of Assumption~\ref{as.u} as described in Appendix~\ref{ap:ab}, and set $a = 1.1, b = 2$, so $\beta_n$ can be determined.} 

The desired distances $\bar{\x}(t)$ are set to $\bar{d}_i = .33 + .25\sin(\pi \omega t)$ for all $i$'s, where $\omega$ is a tunable parameter. The rationale for modeling the $\bar{\x}(t)$ in this way is to capture (1) a stationary case ($\omega = 0$) and a dynamic case where the distances change because of dynamic traffic conditions (and therefore  varying $v(t)$), road changes, etc. Users' feedback comes as a noisy sample of their comfort function, and the noise  is modeled as  a zero-mean Gaussian variable with variance $\sigma = 0.1$. Feedback in this example can come in different ways: it can come at low frequency, if the users hit the break or the accelerator every time they feel too close or too far from the vehicle in front, or it can come at higher frequency, if the users are equipped with heart rate/breathing rate sensors (which can be in smartwatches or incorporated in the seat of the vehicles~\cite{Hao2017}) which may be used as proxies of stress and discomfort. \rev{Here, we assume that the user's feedback is either received within the sampling period (in which case the optimization counter $k$ is the same as the data counter $n$, i.e., $n = k$) or that is received only every $4$ optimization steps (in which case $k = 4n$).}    

\begin{figure}[t!]
\centering
\includegraphics[width = 0.49\textwidth]{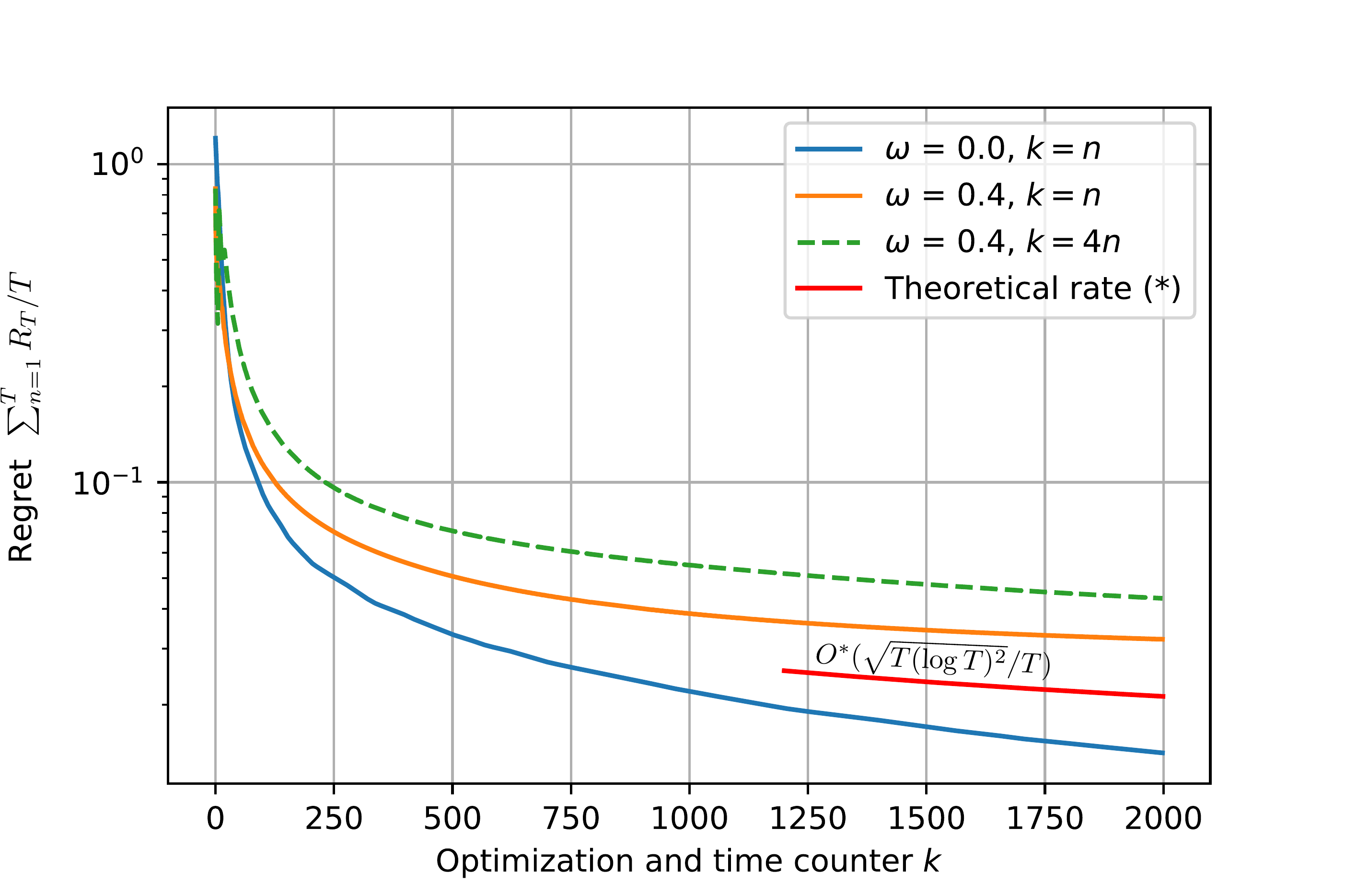}
\caption{\rev{Regret for different $\omega$'s and different number of optimization steps $k$ for each data point $n$. (*) The theoretical rate is for $\omega = 0.0$ and $k = n$. $\delta = 0.1$.} }
\label{fig.1}
\end{figure}  
  
In Figure~\ref{fig.1}, we show the performance of the \ALGO\ algorithm varying $\omega$. On the vertical axis, we plot the regret $\frac{1}{T} R_T$, averaged on \rev{25} different runs of the algorithm. As expected, we observe that when $\omega = 0$, we obtain a no-regret scenario, with $\frac{1}{T} R_T$ eventually going to zero; in this case, it goes to zero  as $O^*(\sqrt{T (\log T)^{2}}/T)$ (note that the learning dimension is $1$ in our example, \rev{and in this case $n = k$}). In the time-varying scenarios instead, there is an asymptotic error bound, thus corroborating the theoretical results. \rev{We also notice that, in the case of $k = 4n$, the algorithm is slower to converge, since the number of feedback points is smaller.}   

\rev{In Figure~\ref{fig.2}, we report how the distance to the platoon leader changes over time for the case $n=k$, $\omega = 0.4$, up to $k = 400$. We notice how the engineering best (Eng. best) for the different following vehicles has larger excursions, while the computed approximate decision variable -- taking into account user's comfort -- is smoother. Roughly speaking, for large distances users feel they are losing the vehicle in front, for smaller distances they feel too close. }

\begin{figure}[t!]
\centering
\includegraphics[width = 0.49\textwidth]{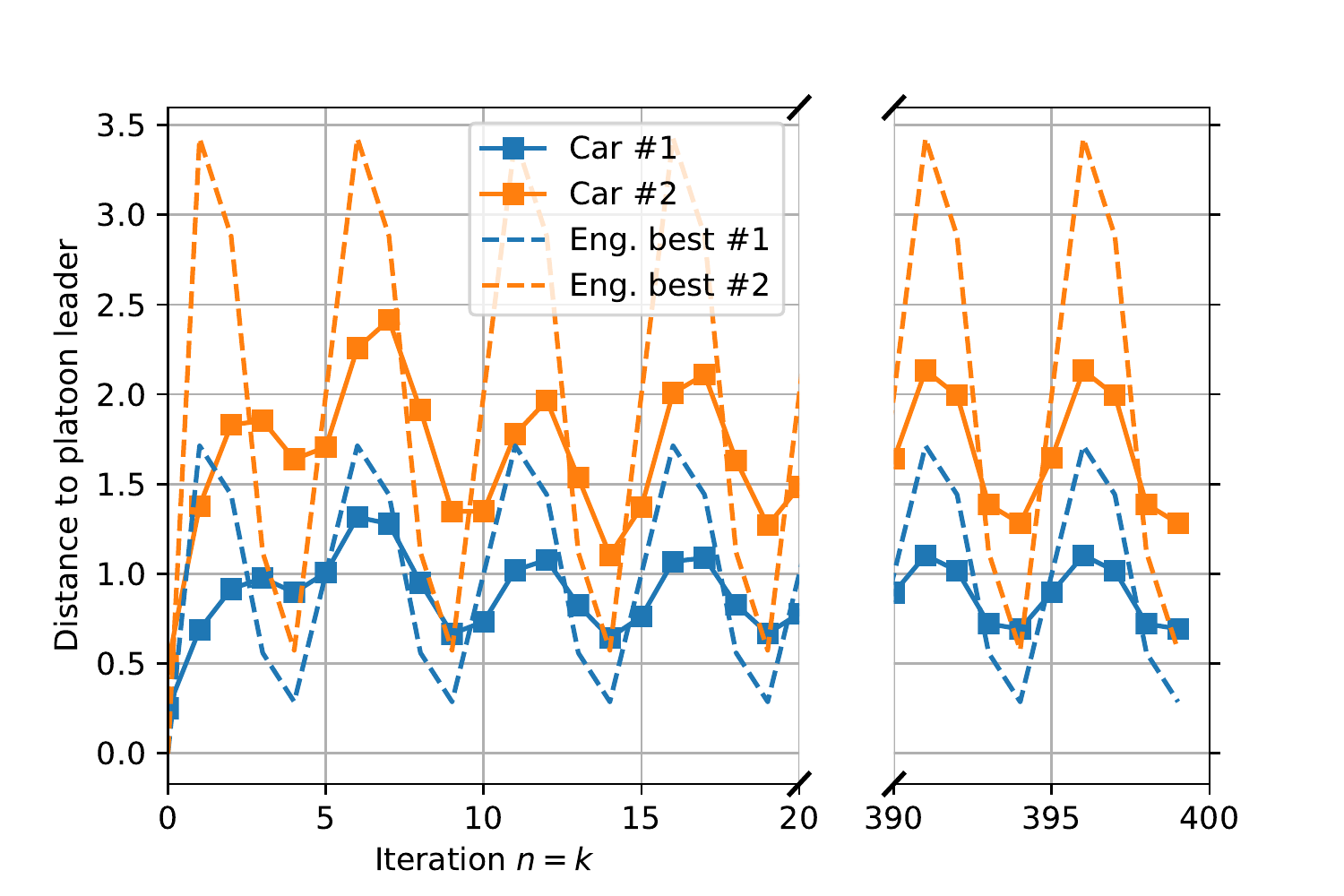}
\caption{\rev{Distance to platoon leader for the considered two following vehicles, in the case $n=k$ and $\omega = 0.4$. } }
\label{fig.2}
\end{figure}  

\rev{In Figure~\ref{fig.3}, we illustrate the learning of user $2$ satisfaction function as more and more feedback comes in, for the case $n=k$, $\omega = 0.4$. In red, we represents the mean, while in green, we represents the 1-$\sigma$ bounds. The grey blobs are the feedback data points. } 

\begin{figure}[t!]
\centering
\includegraphics[width = 0.49\textwidth]{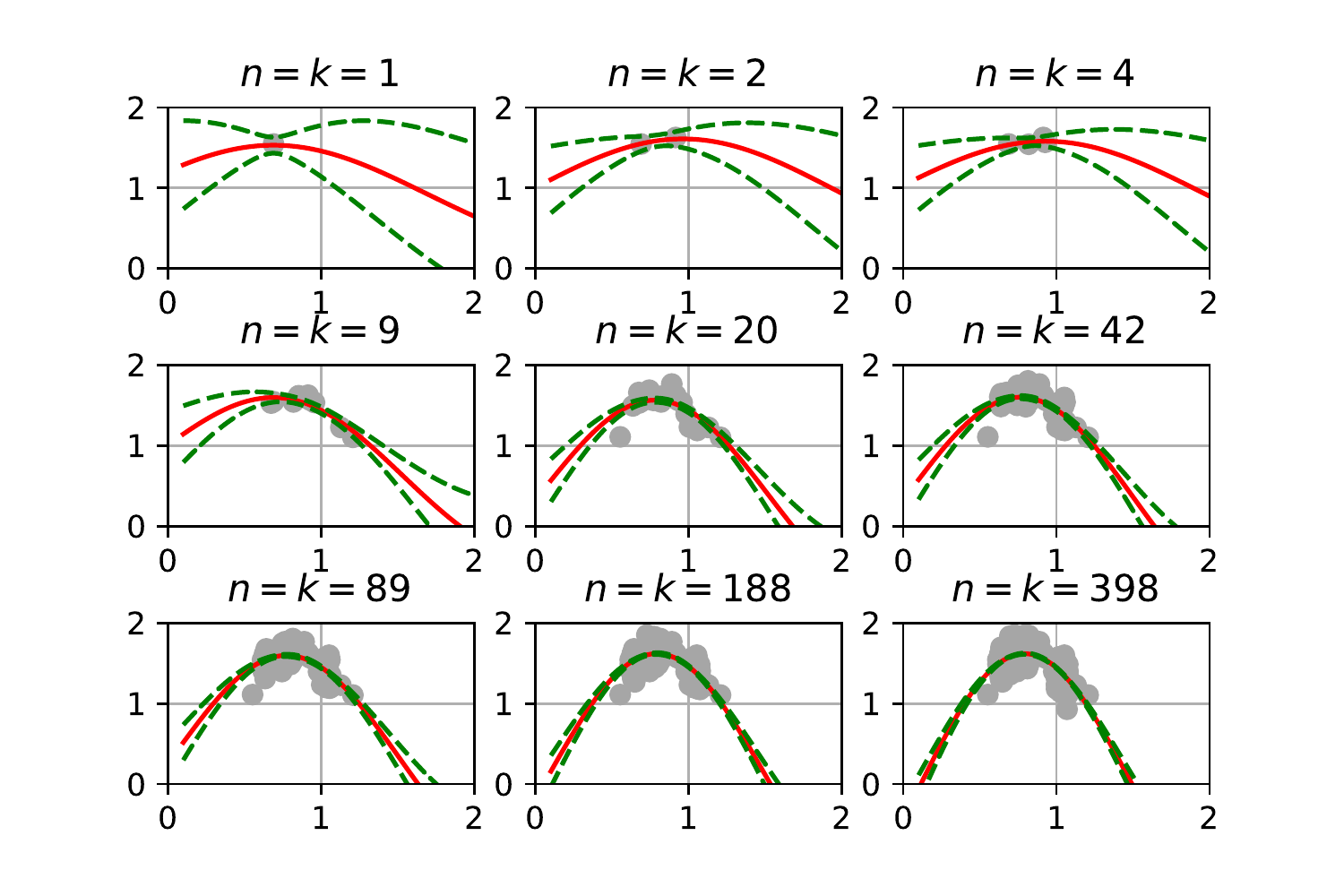}
\caption{\rev{Learning of user $2$ satisfaction function as more and more feedback comes in, for the case $n=k$, $\omega = 0.4$. In red, we represents the mean, while in green, we represents the 1-$\sigma$ bounds. The grey blobs are the feedback data points.}}
\label{fig.3}
\end{figure}

We notice that in this example, the vehicles are supposed to follow the new set-points (i.e., the new inter-vehicle distances) rapidly i.e., faster than \rev{the sampling period} $h$. If this is not the case, the analysis of the algorithm could be extended to include an actuation error;  we plan to investigate this in future research. 


\subsection{Comparison with other approaches}

\arev{
We now compare the proposed algorithm to other possible approaches. In particular, we discuss \emph{(i)} a projected gradient method based on a synthetic one-fits-all model of the user's satisfaction function; \emph{(ii)} a projected zeroth-order method based on the user's feedback at the current and previous point(s) to estimate the gradient of the user's satisfaction function.  

The synthetic model for the  user's satisfaction functions is taken to be $U_i(d_i) = \exp(-\log(d_i)^2/\xi_i^2)/{(\xi_i d_i)}$, with $\xi_i = 0.9$ for both vehicles (to exactly match the parametric model which the GP samples mimic, with a small parameter inaccuracy). We remark that it is in general complicated to obtain accurate one-fits-all models for human preferences, and even small inaccuracies can lead to a bias and a loss of perceived fairness. 

The zeroth-order method utilizes the user's feedback to obtain an estimate of the gradient of the user's satisfaction function at each iteration. We implement both a two-point estimate, where we use the current feedback and the previous one, as well as a four-point estimate, where we use the current feedback and three previous ones. We remark that the implementation of zeroth-order methods, which is in line with~\cite{Duchi2015, Luo2020, Liu2020}, may be in general impractical; in fact, the user may provide feedback infrequently (whereas the zeroth-order method requires feedback at each iteration). While our GP-based method does work seamlessly with intermittent noisy feedback, a zeroth-order method would not.

In Figure~\ref{fig.6}, we see how our \ALGO\ algorithm outperforms the others in terms of average regret. This is reasonable, since the synthetic model has no feedback to learn the true user's satisfaction functions, and the zeroth-order model estimates a noisy gradient (and, likewise, it does not learn the user's functions).

We report in Figure~\ref{fig.7} a metric for the satisfaction of each user. We define this metric as
$$
\textsf{UC}_i(t) := U_i(d_i(t))/ \max_{d\in D} U_i(d) \in [0,1].
$$
The metric $\textsf{UC}_i(t)$, for each user for each time $t$, ``measures'' how satisfied is the user with the current decision variable $d_i(t)$. The results are shown for one run, and they are rolling averages of 5 (for the \ALGO\ algorithm and the synthetic functions) and of 400 (for the zeroth-order algorithms) time points to smooth the signal and capture the overall trend. It can be seen that, not only our \ALGO\ algorithm achieves the best $\textsf{UC}_i(t)$, but it also achieves a similar one for both vehicles. The other competing algorithms create a bias towards one of the vehicles, degrading the perceived fairness and thereby reducing the possible acceptance of the technology~\cite{Linehan2019}. Finally, the zeroth-order methods appear noisy for this example and may generate frustration in a human user.  

\begin{figure}[t!]
\centering
\includegraphics[width = 0.49\textwidth]{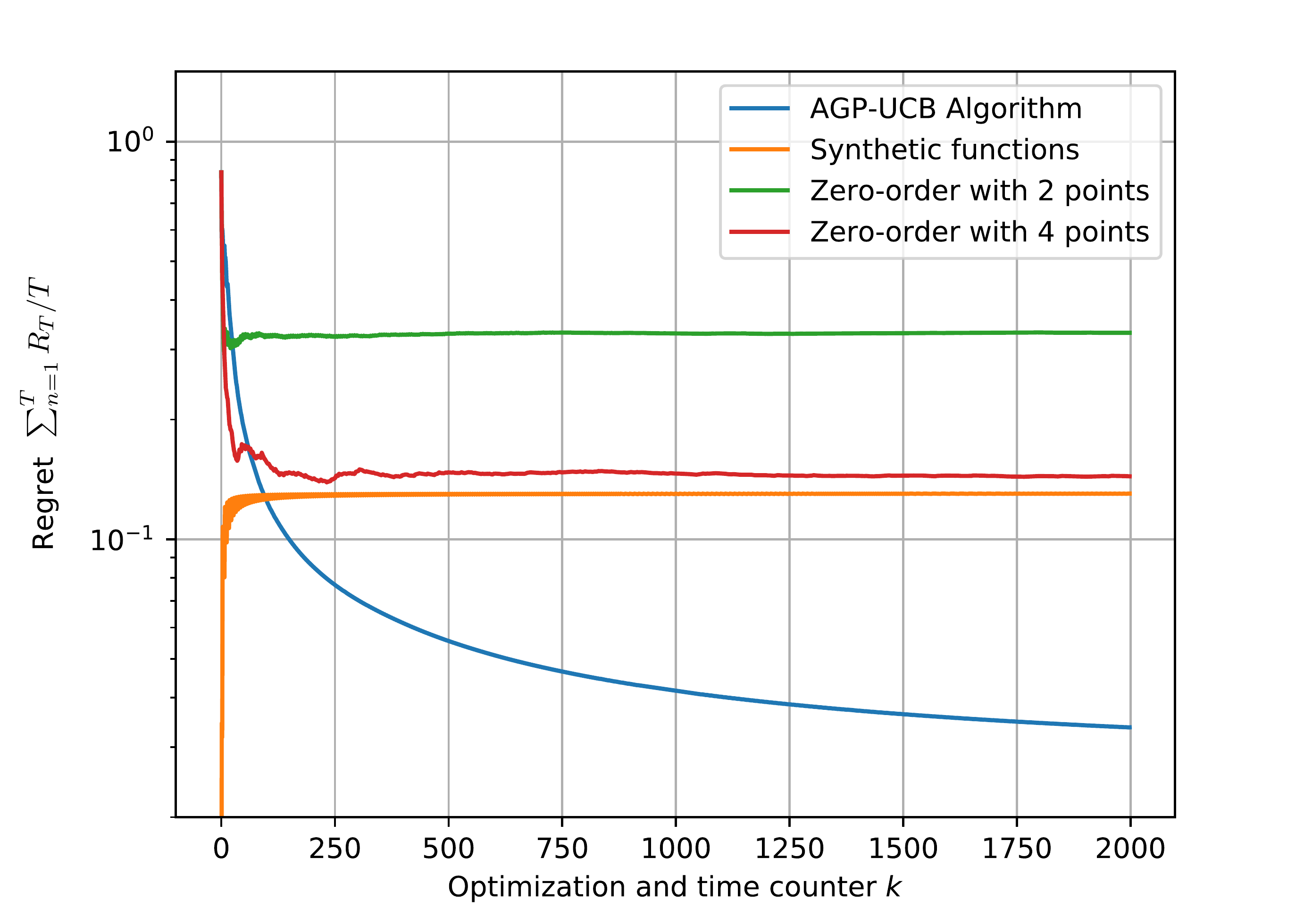}
\caption{\arev{Comparison of the different algorithms in terms of average regret averaged over $25$ runs of the algorithms; $n=k$ and $\omega = 0.4$. } }
\label{fig.6}
\end{figure}  

\begin{figure}[t!]
\centering
\includegraphics[width = 0.49\textwidth, trim=0cm 0cm 0 1cm, clip=on]{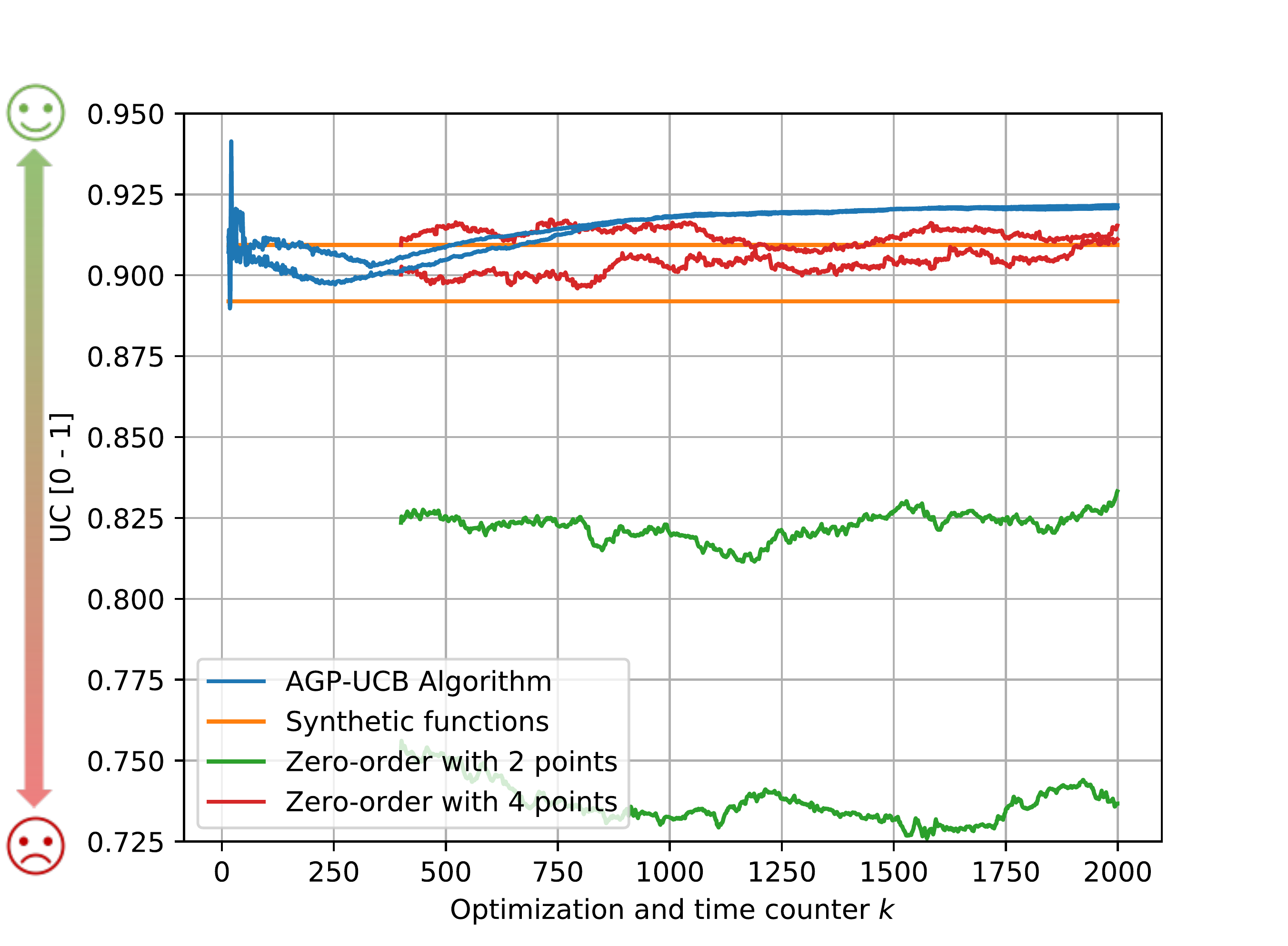}
\caption{\arev{Comparison of the different algorithms in terms of $\textsf{UC}$ (normalized user's safisfaction) for one run of the algorithm; $n=k$ and $\omega = 0.4$. } }
\label{fig.7}
\end{figure}

}

\section{Conclusions}
We developed an online algorithm to solve an optimization problem with a  cost function comprising a time-varying engineering function and a user's  utility function. The algorithm is an approximate upper confidence bound algorithm, and it provably generates a sequence of optimizers that is within a ball of the optimal trajectory, while learning the user's satisfaction function with no-regret. We have illustrated the result with a numerical example derived from vehicle platooning, which offered some additional inspiration for future research. \rev{We believe that our algorithm may offer a wealth of opportunities to devise more ways to integrate users' considerations within the decision making process. }

\vskip3mm

\appendix

\section{Proof of Theorems~\ref{th:Regret}-\ref{th:Total}}\label{ap:th12}
\subsection{Preliminaries}

We start by defining some quantities that will be subsequently  used in the proofs. We rewrite the main algorithmic step \eqref{eq:optimizer} as
\begin{equation}\label{eq:xupdate} 
\x_{n} \approx \arg\max_{\x \in D} \varphi_{nn}(\x) = V(\x; t_n) + \hat{U}_n(\x),
\end{equation}
where we recall that $k=n$; we  simplify the notation from $\varphi_{nn}(\x)$ to $\varphi_{n}(\x)$, since there is no confusion when $k=n$.  

Define the error $e_n$ as the difference between the optimum of the function $\varphi_n(\x)$ and the approximate solution computed via~\eqref{eq:xupdate}, as
\begin{equation}
e_n := \varphi_n^\dag - \varphi_n(\x_{n}).
\end{equation} 
where we use the superscript $\dag$ for  quantities that are related to the optimizer/optimum of $\varphi_n(\x)$, so that the optimum of $\varphi_n(\x)$  is $\varphi_n^\dag = \varphi_n(\x_{n}^\dag)$. 

Define the instantaneous regret $r_n$ as the difference between the optimum of function $f(\x;t_n) = V(\x; t_n) + U(\x)$, which denoted as $f^*_n$, and the approximate optimum $f(\x_{n}; t_n)$ where $\x_{n}$ is computed via~\eqref{eq:xupdate}, as
\begin{equation}
r_n := f_n^* - f(\x_{n}; t_n).
\end{equation} 
Further, define the instantaneous optimal regret $r_n^\dag$ as the difference between $f^*_n$ and the optimum $f(\x_{n}^\dag; t_n)$ where $\x_{n}^\dag$ is the optimizer of~$\varphi_n(\x)$, as
\begin{equation}
r_n^\dag := f_n^* - f(\x_{n}^\dag; t_n).
\end{equation} 

Lemma 5.5 in~\cite{Srinivas2012} is applicable here: pick a $\delta \in (0,1)$ and set $\beta_n = 2 \log(\omega_n/\delta)$, where $\omega_n>0$, $\sum_{n\geq 1} \omega_n^{-1} = 1$. Then, for all $n \geq 1$,
\begin{eqnarray}
|U(\x_{n})-\mu_{n-1}(\x_{n})| &\leq& \sqrt{\beta_n} \sigma_{n-1}(\x_{n}), \label{15}\\
|U(\x_{n}^\dag)-\mu_{n-1}(\x_{n}^\dag)| &\leq& \sqrt{\beta_n} \sigma_{n-1}(\x_{{n}}^\dag). \label{16}
\end{eqnarray}
hold with probability $\geq 1-\delta$. 

Similarly, Lemma~5.6 and Lemma~5.7 of~\cite{Srinivas2012} hold. For Lemma~5.7, we choose a regular discretization $D_n \subset D$, with $D_n \subset [0,r]^d$ and the discretization size $(\tau_n)^d$. Similar to~\cite{Srinivas2012} we have that 
\begin{eqnarray}
\|\x - [\x]_n\|_1 &\leq& r d / \tau_n, \\
\|\x - [\x]_n\|_2 &\leq& r \sqrt{d} /\tau_n, \label{eq-2}\\
\|\x - [\x]_n\|_{\infty} &\leq&  r /\tau_n, \label{eq-inf}
\end{eqnarray}
where $[\x]_n$ is the closest point in $D_n$ to $\x \in D$. We are now able to bound the instantaneous optimal regret $r_n^\dag$.

\begin{lemma}[Extension of Lemma 5.8 of~\cite{Srinivas2012}]\label{lemma:1}
Let Assumptions \ref{as.fun} and \ref{as.u} hold true. Pick a $\delta \in (0,1)$ and define a sequence $\{\omega_n\}_{n\in\mathbf{N}}$ such that $\sum_{n\geq 1} \omega^{-1}_n >0$, $\omega_n >0$. Set the parameter $\beta_n$ as
$$
\beta_{n} = 2\log(4 \omega_{n}/\delta) + 2 d \log(d n^2 b r \sqrt{\log(4 d a /\delta)}),
$$ 
where $a,b,d,r$ are the parameters defined in Assumption~\ref{as.u}. Define the parameters $A_0$ and $A_1$ as
$$
A_0 = \frac{D_g }{b \sqrt{\log(2 d a /\delta)}} + 1,\quad  A_1 = \frac{L}{2 d b^2 {\log(2 d a /\delta)}}.
$$
Then instantaneous optimal regret $r_n^\dag$ is bounded for all $n \geq 1$ as follows:
$$
r_n^\dag \leq 2 \sqrt{\beta_{n}}\, \sigma_{n-1}(\x_{n}^\dag) +A_0/n^2 + A_1/n^4
$$
with probability $\geq 1-\delta$. 
\endstatement
\end{lemma}

\begin{proof}
The proof is a modification of the one of~\cite[Lemma~5.8]{Srinivas2012}; also in this case we use $\delta/2$ in both Lemmas~5.5 and 5.7 of~\cite{Srinivas2012} (which are valid here for discussion above). We report only the parts that are different. Let $\x^*_n$ be the optimizer of $f(\x; t_n)$, i.e., $f^*_n = f(\x^*_n; t_n)$.

By definition of $\x^\dag_{n}$ and optimality, we have that $\varphi_n(\x^\dag_{n}) \geq \varphi_n([\x^*_n]_{n})$. By~\cite[Lemma~5.7]{Srinivas2012}, we have that $\hat{U}_n([\x^*_n]_{n}) + 1/n^2 \geq U(\x^*_{n})$. These two combined yield 
\begin{multline}\label{20}
\hskip-0.3cm\varphi_n(\x^\dag_{n}) \!=\! V(\x^\dag_{n};t_n) \!+\! \hat{U}_n(\x^\dag_{n}) \geq V([\x^*_n]_{n};t_n) \!+\! \hat{U}_n([\x^*_n]_{n}) \geq\\
 V([\x^*_n]_{n};t_n) + U(\x^*_{n}) - 1/n^2.
\end{multline}
Now, by convexity of $-V(\x;t_n)$ and the Lipschitz condition on its gradient (Assumption~\ref{as.fun} and Eq.~\eqref{eq:lip}), 
\begin{equation}
-V(\x;t_n) + V(\y;t_n) \leq d D_g \|\x-\y\|_{\infty} + \frac{L}{2}\|\x-\y\|^2_2.
\end{equation}
Therefore by \rev{substituting} $\x = [\x^*_n]_{n}$ and $\y = \x^*_{n}$, 
\begin{equation}\label{22}
V([\x^*_n]_{n};t_n) \geq V(\x^*_{n};t_n) - (A_0-1)/n^2 \!- \! A_1/n^4
\end{equation}
where, by the definition of $\tau_n$ in~\cite[Lemma~5.7]{Srinivas2012}  and Eq.s~\eqref{eq-2}-\eqref{eq-inf}, we derive the expressions of $A_0$ and $A_1$ reported in Lemma~\ref{lemma:1}. 

Therefore by~\eqref{20} and \eqref{22}, 
\begin{multline}\label{eq:zibby}
\varphi_n(\x^\dag_{n}) \geq 
V(\x^*_{n};t_n) + {U}_n(\x^*_{n}) - A_0/n^2 - A_1/n^4 \\ = f_n^* - A_0/n^2 - A_1/n^4 =: f_n^* - C
\end{multline}
Putting the above results together, and by using~\eqref{16},
\begin{eqnarray}
r_n^\dag &=& f_n^* - f(\x_{n}^\dag; t_n) \leq  \varphi_n(\x^\dag_{n}) + C - f(\x_{n}^\dag; t_n) \nonumber\\ 
&\leq& V(\x_{n}^\dag; t_n)\!+ \!\mu_{n-1}(\x_{n}^\dag) \!+\! \sqrt{\beta_{n}}\sigma_{n-1}(\x_{n}^\dag)\!+\!  C \!+\!  \nonumber\\ && \quad - V(\x_{n}^\dag; t_n)\!-\! U(\x_{n}^\dag; t_n) \nonumber \\
&\leq& V(\x_{n}^\dag; t_n)\!+\! \sqrt{\beta_{n}}\sigma_{n-1}(\x_{n}^\dag) \!+\!  C \!+\!  \nonumber\\ && \, + |\mu_{n-1}(\x_{n}^\dag)- U(\x_{n}^\dag; t_n)|  -V(\x_{n}^\dag; t_n) \nonumber \\ &\leq& 2 \sqrt{\beta_{n}}\, \sigma_{n-1}(\x_{n}^\dag) +C \label{24},
\end{eqnarray}
from which the claim. 
\end{proof}

\begin{lemma}\label{lemma:regret}
Under the same assumptions and definitions of Lemma~\ref{lemma:1}, the instantaneous regret $r_n$ is bounded for all $n\geq 1$ as follows:
$$
r_n \leq 2 \sqrt{\beta_{n}}\, \sigma_{n-1}(\x_{n}) +A_0/n^2 + A_1/n^4+ e_{n}
$$
with probability $\geq 1-\delta$.
\endstatement
\end{lemma}
\begin{proof}
By definition of $e_n$, by calling $C = A_0/n^2 + A_1/n^4$, and by~\eqref{24}, we have 
\begin{eqnarray*}
r_n &=& f_n^* - f(\x_{n}; t_n)\\ &\leq& \varphi_n(\x_{n}) +e_n +C - f(\x_{n}; t_n)\\
&\leq& V(\x_{n}; t_n)+ \sqrt{\beta_{n}}\sigma_{n-1}(\x_{n})+ C+\\ && \quad +e_n +|\mu_{n-1}(\x_{n})- U(\x_{n}; t_n)|  -V(\x_{n}; t_n)\\ &\leq& 2 \sqrt{\beta_{n}}\, \sigma_{n-1}(\x_{n}) +C +e_n,
\end{eqnarray*}
where we have use~\eqref{15} in the last inequality. 
\end{proof}

We now move to the analysis of the error sequence $\{e_n\}_{n \geq 1}$.  

\begin{lemma}\label{lemma:3}
Let Assumptions~\ref{as.fun} \rev{until} \ref{as.method} hold true. Define the learning rate error $\ell_n$ as
$$
\ell_n = \max_{\x \in D} |\hat{U}_n(\x) - \hat{U}_{n-1}(\x)|.
$$
Then, the error sequence $\{e_n\}_{n \geq 1}$ is upper bounded as
\begin{equation*}
e_n \leq  \eta^{n-1} e_1 + \sum_{z=1}^{n-1} 2\, \eta^{z} (\Delta_{n-z+1}+ \ell_{n-z+1}).
\end{equation*}
\endstatement
\end{lemma}

\begin{proof}
From the definition of $e_n$ we obtain, 
\begin{equation}
e_n = \varphi^\dag_n - \varphi_{n}(\x_{n}) \leq \eta (\varphi_n^\dag - \varphi_n(\x_{n-1})),
\end{equation}
where the inequality is due to Assumption~\ref{as.method} on the convergence rate of method $\mathcal{M}$. In addition,
\begin{multline*}
\varphi^\dag_n - \varphi_n(\x_{n-1}) = \underbrace{\varphi_n^\dag\!-\!\varphi_{n-1}^\dag}_{(A)} +
\underbrace{\varphi_{n-1}(\x_{n-1})\!-\!\varphi_n(\x_{n-1})}_{(B)} + \\ + e_{n-1}.
\end{multline*}
By algebraic manipulations, the error $(A)$ can be written as,
\begin{equation*}
(A) = (\varphi_n^\dag - \varphi_{n-1}(\x_{n}^\dag))-(\varphi_{n-1}^\dag -\varphi_{n-1}(\x_{n}^\dag)).
\end{equation*}
The first term in the right hand side can be written as
\begin{multline*}
\varphi_n^\dag - \varphi_{n-1}(\x_{n}^\dag) = V(\x_{n}^\dag; t_n) - V(\x_{n}^\dag; t_{n-1}) + \\ \hat{U}_n(\x_{n}^\dag) - \hat{U}_{n-1}(\x_{n}^\dag) \leq \Delta_n + \ell_n.
\end{multline*}
The second term is $-(\varphi_{n-1}^\dag -\varphi_{n-1}(\x_{n}^\dag)) \leq 0$, by optimality.  \rev{Therefore},
$(A) \leq \Delta_n + \ell_n$.
%
%
Furthermore, by Assumption~\ref{as.tv},
\begin{eqnarray*}
(B) &=& V(\x_{n-1}; t_{n-1}) + \hat{U}_{n-1}(\x_{n-1}) + \\ && \qquad\qquad - V(\x_{n-1}; t_{n}) - \hat{U}_{n}(\x_{n-1}) \\ 
&\leq& \Delta_n + \ell_n.
\end{eqnarray*} 

Therefore, a bound on $e_n$ can be derived as
\begin{align*}
e_n &\leq \eta(2 \Delta_n + 2\ell_n  + e_{n-1}) \\
&\!=\! \eta^{n-1} e_1 + \sum_{z=1}^{n-1} 2\, \eta^{z} (\Delta_{n-z+1}+ \ell_{n-z+1}).
\end{align*}
From which the claim follows. 
\end{proof}

In the next lemma, we characterize how the learning rate error $\ell_n$ evolves. 

\begin{lemma}[Cumulative learning rates] \label{lemma:lr} With the same assumptions of Lemma~\ref{lemma:3} and additionally Assumption~\ref{as.behave}, we have that the cumulative learning rate error $\sum_{n=2}^T\ell_n$ is bounded as
$
\sum_{n=2}^T \ell_n \leq O^*(1).
$
\label{lemma:l}
\endstatement
\end{lemma}
\begin{proof}
First we look at the mean $\mu_n(\x)$ and derive a bound on the difference between $\mu_n(\x)$ and $\mu_{n-1}(\x)$. We proceed in similar fashion for the variance $\sigma_n^2(\x)$. Finally, we put the results together and we prove the lemma. We use the theory of Kernel ridge regression (KRR) estimation to map mean and variance as optimizers of carefully constructed optimization problems over Hilbert spaces, see~\cite{Rasmussen2006,Yang2017}. 

Let the covariance $k(\x,\x')$ be the positive definite kernel associated with the Gaussian process, let $\{\phi_i(\x)\}$ be the eigenfunctions of said kernel $k$ and $\{\lambda_i\}$ be the associated eigenvectors. The posterior mean for GP regression~\eqref{eq:mean} can be obtained as the function which minimizes an appropriately defined functional defined over Hilbert spaces:
\begin{equation}\label{krr}
\mu_n = \arg \min_{\mu\in \mathcal{H}} \Big\{\frac{1}{2} \|\mu\|_{\mathcal{H}}^2 + \frac{1}{2\sigma^2}\sum_{i=1}^n (y_i - \mu(\x_i))^2\Big\},
\end{equation}
where the norm $\|\cdot\|_{\mathcal{H}}$ is the reproducing kernel Hilbert space (RKHS) norm associated with kernel $k$. Under Assumption~\ref{as.behave} the eigenfunctions form a complete orthonormal basis and the true mean can be written as $\mu_{\infty}(\x) = \sum_{i=1}^\infty \eta_i \phi_i(\x)$, and $\mu_{n}(\x) = \sum_{i=1}^\infty \zeta_i \phi_i(\x)$.  By rewriting~\eqref{krr} in terms of the coefficients $\zeta_i$ (see Eq.~(7.5) in~\cite{Rasmussen2006}, with our notation), we arrive at the optimal solution:
\begin{equation}
\zeta_i = \frac{\lambda_i}{\lambda_i + \sigma^2/n} \eta_i.
\end{equation}
The same holds for $\mu_{n-1}(\x)$ for some coefficients $\zeta_i'$. This implies
\begin{equation*}
|\mu_n(\x)-\mu_{n-1}(\x)| = \Big|\sum_{i=1}^\infty \frac{\lambda_i\sigma^2\, \eta_i}{(n\lambda_i + \sigma^2)((n-1)\lambda_i + \sigma^2)}\phi_i(\x)\Big|.
\end{equation*}
Therefore, since $\lambda_i >0$, $|\mu_n(\x)-\mu_{n-1}(\x)| = O(1/n^2)$ and 
%
by looking at the cumulative difference,
\begin{equation}\label{eq:mm}
\sum_{n=2}^T |\mu_n(\x)-\mu_{n-1}(\x)| = O(1).
\end{equation}

We move on to the variance $\sigma^2_n(\x)$. Here we make use of the fact that the variance can be related to the bias of a noise-free KRR estimator (see~\cite{Yang2017}). In particular, define $\tilde{k}_n^{\x'}(\x)$ as $\k_n(\x)^\transp(\K_n + \sigma^2 \I_n)\k_n(\x')$. From~\eqref{eq:cov}, the convergence rate of the variance $\sigma^2_n(\x)$ is the convergence rate of $\tilde{k}_n^{\x'}(\x)$. Fix $\x'$, from the form of $\tilde{k}_n^{\x}(\x')$, one can interpret $k_n(\x_i,\x')$ as noise-free measurements and proceed as in the case of the mean, deriving
\begin{equation}\label{krr1}
\tilde{k}_n^{\x'} = \arg \min_{\kappa\in \mathcal{H}} \Big\{\frac{1}{2} \|\kappa\|_{\mathcal{H}}^2 + \frac{1}{2\sigma^2}\sum_{i=1}^n (k_n(\x_i,\x') - \kappa(\x_i))^2\Big\}.
\end{equation}
Now, by rewriting the problem in terms of the coefficients of the eigenfunction expansion (so that $\tilde{k}^{\x'}_{\infty}(\x) = \sum_{i=1}^\infty \xi_i \phi_i(\x)$, and $\tilde{k}_n^{\x'}(\x) = \sum_{i=1}^\infty \theta_i \phi_i(\x)$), then
\begin{equation}
\theta_i = \frac{\lambda_i}{\lambda_i + \sigma^2/n} \xi_i.
\end{equation}
With this in place,
\begin{multline}\label{eq.den}
|\sigma_n(\x) - \sigma_{n-1}(\x)| = |\sqrt{\sigma_n^2(\x)} - \sqrt{\sigma^2_{n-1}(\x)}| = \\
\Big|\frac{{\sigma_n^2(\x)} - {\sigma^2_{n-1}(\x)}}{\sqrt{\sigma_n^2(\x)} + \sqrt{\sigma^2_{n-1}(\x)}}\Big| = \frac{\big|\tilde{k}_n^{\x}(\x) - \tilde{k}^{\x}_{n-1}(\x)\big|}{\sqrt{\sigma_n^2(\x)} + \sqrt{\sigma^2_{n-1}(\x)}}.
\end{multline}
The numerator of~\eqref{eq.den} can be treated as done for the mean and is $O(1/n^2)$, while the denominator is $O(1)$ as $n \to \infty$, therefore
\begin{equation}\label{eq.kk}
\sum_{n=1}^T  |\sigma_n(\x) - \sigma_{n-1}(\x)| = O(1).
\end{equation}
By \rev{combining}~\eqref{eq:mm} and \eqref{eq.kk}, one obtains
\begin{multline*}
\sum_{n=1}^T  \ell_n \leq \sum_{n=1}^T |\mu_n(\x) - \mu_{n-1}(\x)| + \\ \sqrt{\beta_T} |\sigma_n(\x) - \sigma_{n-1}(\x)| \leq O^*(1),
\end{multline*}
from which the claim follows.
\end{proof}

\subsection{Proof of Theorem~\ref{th:Regret}}
\begin{proof}
By Lemma~\ref{lemma:regret}, we have that with probability greater than $1-\delta$
$$
r_n \leq 2 \sqrt{\beta_{n}}\, \sigma_{n-1}(\x_{n}) +A_0/n^2 + A_1/n^4+ e_{n}
$$
whereas $e_n$ can be bounded by Lemma~\ref{lemma:3} and Lemma~\ref{lemma:l}. The bound on the variance, as derived in Lemma~5.4 of \cite{Srinivas2012} is valid here as well, and in particular with probability greater than $1-\delta$, 
\begin{equation}
\sum_{n=1}^T 4 {\beta_{n}}\, \sigma_{n-1}^2(\x_{n}) \leq {C_1 \beta_T \gamma_T}, \quad \forall T \geq 1,
\end{equation}
with $C_1 = 8/\log(1+\sigma^{-1})$, so that by Cauchy-Schwarz:
\begin{equation}
\sum_{n=1}^T 2 \sqrt{\beta_{n}}\, \sigma_{n-1}(\x_{n}) \leq \sqrt{C_1 T \beta_T \gamma_T}, \quad \forall T \geq 1.
\end{equation}
Therefore, for all $T\geq 1$ 
\begin{equation}
R_T = \sum_{n=1}^T r_n \leq \sqrt{C_1 T \beta_T \gamma_T} + A_0 \frac{\pi^2}{6} + A_1 \frac{\pi^4}{90} + \sum_{n=1}^T e_{n},
\end{equation}
where we have used: $\sum 1/n^2 = \pi^2/6$ and $\sum 1/n^4 = {\pi^4}/{90}$. As in~\cite{Srinivas2012}, we use the crude bound ${\pi^4}/{90} < \pi^2/6 < 2$, and define 
\begin{equation}
C_2 = A_0 \frac{\pi^2}{6} + A_1 \frac{\pi^4}{90} < \frac{2 D_g }{b \sqrt{\log(2 d a /\delta)}}+\frac{2 L}{2 d b^2 {\log(2 d a /\delta)}} + 2,
\end{equation} 
so that,
\begin{equation}
R_T = \sum_{n=1}^T r_n \leq \sqrt{C_1 T \beta_T \gamma_T} + C_2 + \sum_{n=1}^T e_{n},
\end{equation}

As for the error $e_n$ term, by Lemma~\ref{lemma:3} and Lemma~\ref{lemma:l} 
\begin{equation}
\sum_{n=1}^T e_n \leq \frac{e_1}{1-\eta} + 2 \sum_{n=1}^T \sum_{z=1}^{n-1} \eta^z \Delta_{n-z+1} + 2 \sum_{n=1}^T \sum_{z=1}^{n-1} \eta^z \ell_{n-z+1}.
\end{equation}
We bound the right-hand terms as follows
\begin{equation}
\sum_{n=1}^T \sum_{z=1}^{n-1} \eta^z \Delta_{n-z+1} \leq \Delta  \sum_{n=1}^T \eta \frac{1-\eta^{n-1}}{1-\eta} \leq \Delta \eta T/(1-\eta)
\end{equation}
and
\begin{equation}
\sum_{n=1}^T \sum_{z=1}^{n-1} \eta^z \ell_{n-z+1} =  \sum_{n=1}^T \eta^{n}  \sum_{z=1}^{T-n} \ell_{z+1} \leq \frac{1}{1-\eta}\,O^*(1),  
\end{equation}
since for Lemma~\ref{lemma:lr}, $\sum_{n=2}^T \ell_{n} = O^*(1)$; and the claim follows.

Note that the choice of $\beta_n$ in the statement is the same choice of Lemma~\ref{lemma:1} (and subsequent ones), with the special selection of $\omega_n =  n^2 \pi^2/6$. Note that the error term $e_1/(1-\eta)$ is $O(1)$ and it can be incorporated into the $O^*(1)$ term. 
\end{proof}

\subsection{Proof of Theorem~\ref{th:Total}}

\begin{proof}
We use~\rev{\cite{Seeger2008}, or equivalently~\cite[Theorem~5]{Srinivas2012}} to bound the information gain $\gamma_T$ for squared exponential kernels as $\gamma_T = O((\log T)^{d+1})$, from which the claim is derived.
\end{proof}

\section{Proof of Eq.~\eqref{eq:pl}}\label{ap:ka}

\begin{proof}
For completeness and because it is not straightforward, we report here the steps to extend the results of~\cite{Karimi2016} to Eq.~\eqref{eq:pl}. First, we notice that in~\cite{Karimi2016}, unconstrained gradient methods and proximal-gradient methods are considered. Here we are in the proximal-case, with $g(\x)$ representing the indicator function of the compact set $D$. For short-hand notation, we let $\psi_{nk} = -\varphi_{nk}$ and $\nabla = \nabla_{\x}$. In this context, iteration~\eqref{eq:optimizer:pg} can be equivalently written as
\begin{multline}
\!\!\!\!\x_{k} \!=\! \arg\min_{y} \!\left\{ \nabla\psi_{nk}(\x_{k-1})^\transp (\y - \x_{k-1}) +\right. \\  \frac{1}{2 \alpha} \|\y - \x_{k-1} \|^2  + g(\y)\Big\},
\end{multline} 
where $g(\cdot)$ is the indicator function for $D$. Our PL inequality~\eqref{pl:our}, valid for $\x \in D$,  is precisely the same as in~\cite[Eq. (12)-(13)]{Karimi2016} when $g(\cdot)$ is the indicator function; in fact, the condition in~\cite{Karimi2016} pertains to a $\mathcal{D}(\x, c)$ defined as
\begin{multline}\label{their:pl}
\mathcal{D}'(\x, c) := - 2 c \min_{\y} \left\{\nabla\psi_{nk}(\x)^\transp (\y - \x) + \frac{c}{2} \|\y - \x \|^2 \right. \\ \left. + g(\y) - g(\x) \right\}.
\end{multline} 

Now, by using the Lipschitz continuity of the gradient of $\psi_{nk}$, we have that
\begin{eqnarray*}
\psi_{nk}(\x_{k}) &=& \psi_{nk}(\x_{k}) + g(\x_k)\\
&\leq& \psi_{nk}(\x_{k-1}) + \nabla\psi_{nk}(\x_{k-1})^\transp (\x_{k} - \x_{k-1}) +\\ && \qquad\qquad \frac{\Theta}{2}\|\x_{k} - \x_{k-1}\|^2 + g(\x_k) \\
&\leq& \psi_{nk}(\x_{k-1}) + \nabla\psi_{nk}(\x_{k-1})^\transp (\x_{k} - \x_{k-1}) +\\ && \qquad\qquad  \frac{1}{2\alpha}\|\x_{k} - \x_{k-1}\|^2 + g(\x_k) \\
&\leq& \psi_{nk}(\x_{k-1}) -\frac{\alpha}{2} \mathcal{D}(\x_{k-1}, 1/\alpha) \\
&\leq& \psi_{nk}(\x_{k-1}) -\alpha \kappa  (\psi_{nk}(\x_{k-1}) - \psi_{nk}^*).
\end{eqnarray*} 
In particular, we have used the following line of reasoning. First line: $g(\x_k) = 0$, since $\x_k\in D$; second line: Lipschitz property; third line: upper bound true for any $\alpha \leq 1/\Theta$; fourth line: definition of $\mathcal{D}$; fifth line: PL property~\eqref{pl:our}.

By adding and subtracting $\psi_{nk}^*$ to the last inequality and rearranging, the Eq.~\eqref{eq:pl} is proven. 
\end{proof}

\section{Proof of Theorem~\ref{th:vc}}\label{ap:th3}

\begin{proof}
The error term $G_T$ can be written as
\begin{equation}
G_T = 2 \sum_{n=1}^T \sum_{z=1}^{n-1} \eta^z \Delta_{n-z+1} = 2 \sum_{n=1}^T \eta^{n}  \sum_{z=1}^{T-n} \Delta_{z+1}.
\end{equation}
If $\Delta_{k} \to 0$ as $k \to \infty$, then $\sum_{z=1}^{T-n} \Delta_{z+1}$ is sublinear in $T$, i.e., $\sum_{z=1}^{T-n} \Delta_{z+1} = o(T)$. Therefore 
\begin{equation}
G_T \leq \frac{2}{1-\eta} o(T),
\end{equation}
which is sublinear in $T$ and the no-regret claim is proven. 

To obtain a result in terms of regret that is indistinguishable from the static case of\cite[Theorem~2]{Srinivas2012}, $G_T$ needs to scale at worst as $O(\sqrt{T})$, which is now the leading term of the regret. This is the case if $\Delta_{k} \leq O(1/\sqrt{k})$, for which
\begin{equation}
G_T \leq \frac{2}{1-\eta} O(\sqrt{T}),
\end{equation}
which leads to
\begin{equation*}
\mathbf{Pr}\Big\{R_T \leq \sqrt{C_1 T \beta_T \gamma_T} + C_2  + O^*(1)+ O(\sqrt{T}) \Big\} \geq 1-\delta. 
\end{equation*}
Since the term $C_2  + O^*(1)+ O(\sqrt{T})$ grows slower than $\sqrt{C_1 T \beta_T \gamma_T}$ in $T$, it can be omitted in a $O^*$ analysis and the claim is proven. 
\end{proof}

\section{Estimating the constants $a$ and $b$.}\label{ap:ab}

\rev{
The constants $a$ and $b$ in Assumption~\ref{as.u} need to be estimated for different kernels; however, the procedure is rather straightforward and empirical bounds for $a$ and $b$ are not difficult to obtain. In the following, we provide an example of procedure for the squared exponential kernel. 

Consider then the squared exponential kernel
$
k(\x, \x') = \exp\left(-\frac{1}{2 \ell^2} \|\x-\x'\|^2\right)
$
and a Gaussian Process with sample paths $U \sim$GP$(0, k(\x,\x'))$. The derivative is a linear operator, therefore derivatives of GPs are still GPs, and specifically
$$
\frac{\partial U}{\partial x_j} \sim \textrm{GP}(0, k'_j(\x,\x')), \, k_j'(\x,\x') = \frac{\partial}{\partial {x}_j}\frac{\partial}{\partial {x}'_j} k(\mathbold{x}, \mathbold{x}'),
$$ 
where the subscript $j$ stands for the $j$-th component. For a squared exponential kernel, then 
$$
k'_j(\x,\x') = 
k(\x,\x')\left[1 - (x_j-x_j')^2/{\ell^2}  \right]/{\ell^2},
$$
see for example \cite{Solak2003}, and in particular, 
$
k'_j(\x,\x) = {1}/{\ell^2}.
$

Now, take any GP, say again $U \sim $GP$(0, k(\x,\x'))$. For any given $\epsilon>0$ there exists a constant $C_{\epsilon}>0$, such that any sample path $U(\x)$ will be bounded by the following: 
$$
\textbf{Pr}\left\{\sup_{\x \in D} |U(\x)| > M \right\} \leq C_{\epsilon} \exp\left(- \frac{1}{2} \frac{M^2}{ \sigma^2+\epsilon}\right).
$$
where $\sigma^2 = \sup_{\x \in D} k(\x,\x)$, see~\cite[Eq.~(2.33)]{Azaies2009}, or equivalently~\cite[Theorem~5]{Ghosal2006}. Therefore, for the derivative (if exists and bounded as in the squared exponential case), 
$$
\textbf{Pr}\left\{\sup_{\x \in D} \left|\frac{\partial U}{\partial x_j}\right| > M \right\} \leq C_{\epsilon} \exp(- M^2 / (2/\ell^2 + 2\epsilon)).
$$
Comparing this with the Assumption~\ref{as.u}, implies that $b = \sqrt{(2/l^2 + 2\epsilon)}$ and $a = C_{\epsilon}$.

At this point, take $b = \sqrt{(2/\ell^2 + 2\epsilon)}$. In our experiment simulations $\ell = 1$ and we take $\epsilon = 1$ as well, so that $b=2$. 

Compute $C_{\epsilon}$ empirically, at the beginning, by drawing sample paths of the derivative and checking their suprema for different level of $M$ and their empirical frequency. In our case $C_{\epsilon} \approx 1.005$ and we set $a = 1.1$ to be on the safe side. Note that theoretical bounds on $C_{\epsilon}$ exist, but they may be loose, so computing $C_{\epsilon}$ empirically is a good strategy. }

\arev{As discussed, while the value for $b$ is imposed, the value of $a$ is estimated, and in general the estimated $\hat{a}\leq a = C_{\epsilon}$. We see next why this is not an issue.}

\section{Error in the estimate of $a$}

\arev{
Suppose now that the estimated $\hat{a} \leq a$ (this without loss of generality). Given the bound $\hat{a}$, one can pick $\beta_n$ in Theorem~\ref{th:Regret}, as 
$$
\hat{\beta}_{n} = 2 \log(2 n^2 \!\pi^2 \!/(3\delta)) + 2 d \log(d n^2 b r \sqrt{\log(4 d \hat{a} /\delta)}),
$$ 
and this is the only modification for our algorithm. For the convergence analysis, this new $\hat{\beta}_{n}$ changes our Lemma~\ref{lemma:1}. In particular, $A_0$ and $A_1$ are modified. By following closely the proof of Lemma~\ref{lemma:1}, and of \cite[Lemma~5.7]{Srinivas2012}, we can define a discretization size $\hat{\tau}_n = d n^2 b r \sqrt{ \log(2 d \hat{a}/\delta)}$, and by algebraic computations, the new $A_0$ and $A_1$ with the modified $\hat{\beta}_{n}$ read,
\footnotesize
$$
A_0 = \frac{D_g }{b \sqrt{\log(2 d \hat{a} /\delta)}} + \frac{\sqrt{\log(2 d a /\delta)}}{\sqrt{\log(2 d \hat{a} /\delta)}}, \quad  A_1 = \frac{L}{2 d b^2 {\log(2 d \hat{a} /\delta)}}. 
$$ 
\normalsize
Since the new $A_0$ and $A_1$ are larger than the original ones for $\hat{a} \leq a$, then the \emph{constant} term $C_2$ in Theorem~\ref{th:Regret} is larger than the original one. This however does not change the asymptotics of the average regret. 

In addition, in Theorem~\ref{th:Regret}, the term containing $\beta_T$ would be substituted with $\hat{\beta}_T$, which is smaller for the case $\hat{a} \leq a$. Once again, this does not change the asymptotics but only a multiplicative constant. 

  
Summarizing, while having a good bound for $a$ helps to obtain a good $\beta_n$, and therefore a good trade-off between mean and covariance (i.e., exploitation and exploration), even rough bounds for $a$ deliver the same asymptotical behavior for the average regret.
}


\bibliographystyle{ieeetr}        
\bibliography{PaperCollection00}

\end{document}